\documentclass[a4paper,12pt]{article}
\usepackage{complexity}
\usepackage{lineno}
\usepackage{lmodern}
\usepackage{newtxtext}
\usepackage{xcolor}
\usepackage[T1]{fontenc}
\usepackage{graphicx}
\usepackage{textgreek}
\usepackage{pstricks,pspicture}
\usepackage{float}
\usepackage{amsthm,amsmath,amssymb}
\usepackage{mathrsfs,amssymb,amsmath,amsthm,amstext,amsfonts}
\usepackage{pst-plot}
\usepackage{epsfig}
\usepackage{setspace}
\usepackage{float}
\usepackage{rotating}
\usepackage{tikz}
\usepackage{enumitem}
\usepackage{comment}
\usepackage[left=2cm,top=2cm, bottom=2cm,right=2cm]{geometry}

\DeclareMathOperator{\hullc}{Hull_{cc}}

\DeclareMathOperator{\intv}{I_{cc}}
\DeclareMathOperator{\hn}{hn_{cc}}
\DeclareMathOperator{\con}{con_{cc}}

\newtheorem{theorem}{Theorem}

\newtheorem{definition}[theorem]{Definition}
\newtheorem{lemma}[theorem]{Lemma}

\newtheorem{proposition}[theorem]{Proposition}
\newtheorem{remark}[theorem]{Remark}
\newtheorem{notation}[theorem]{Notation}

\newtheorem{problem}{Problem}

\title{Computing the Exchange Number in Graphs with respect to Cycle Convexity}\date{}
\author{Revathy S.Nair$^{a}$\footnote{revathyrahulnivi@gmail.com}
\and Bijo S. Anand$^{b}$\footnote{Corresponding Author Email: bijos\_anand@yahoo.com} \and Julliano R. Nascimento$^{c}$\footnote{jullianonascimento@ufg.br}
 \\\\
 $^{a}$ \small Department of Mathematics, Mar Ivanios College, University of Kerala, \\\small Thiruvananthapuram, India\\
$^{b}$\small Department of Mathematics, Sree Narayana College, Punalur, Kerala \\$^{c}$\small Instituto de Informática, Universidade Federal de Goiás, Goiânia, GO, Brazil
}

\begin{document}
\maketitle

\begin{abstract}
    Given a graph $G$, a subset $S \subseteq V(G)$ is \textit{cycle convex}, if for any vertex $v \in V(G) \setminus S$, the induced subgraph, $G[S \cup \{v\}]$ cannot form a cycle containing the vertex $v$. The \textit{exchange number} of $G$, denoted by $e_{cc}(G)$ is the maximum cardinality of an $\textit{$E$-independent}$ set of $G$. This paper studies the computational complexity of determining the exchange number of graphs and provides exact values for some graph classes. Given a graph $G$ and a positive integer $k$, we show that deciding whether $e_{cc}(G) \geq k$ is NP-complete even if $G$ is a $K_5$-free graph. In contrast, we characterize all $n$-vertex graphs $G$ with exchange number $n-1$ and obtain closed formulas for chordal graphs $G$ whose blocks lie in a single chain, which leads to polynomial-time algorithms for computing $e_{cc}(G)$. We also establish a lower bound for the exchange number of the Cartesian product of general graphs and by using the results of Anand et al. \cite{bijo2}, we derive an explicit formula for the exchange number of strong and lexicographic graph products.   
\end{abstract}


\noindent{\small {\bf AMS Subj. Class:} 05C69, 05C76, 05C85.}

\section{Introduction}\label{section_introduction}
The idea of convexity spaces has been researched in graphs for many decades and numerous graph convexities have been tracked over the years. For a graph $G$, a  finite \emph{convexity space} is a pair $(X,\mathcal{C})$ where, $X$ is a finite set and $\mathcal{C}$, a collection of subsets of $X$ such that; $\emptyset , X \in \mathcal{C}$ and $\mathcal{C}$ is closed under intersections. The elements in $\mathcal{C}$ are said to be the \textit{convex sets}; see $\cite{van22}$.

Many convexity structures in association with the vertex set of graphs have attracted considerable attention in the literature. The path convexities are the most natural convexities in graphs defined in terms of a family of paths $\mathcal{P}$. A vertex set $S \subseteq V(G)$ is $\mathcal{P}$-\emph{convex}, if $S$ includes all vertices of every path in $\mathcal{P}$ between any two vertices of $S$. A detailed study of different types of path convexities can be found in \cite{pelayo-2015}. \emph{Geodesic convexity}, one of the well-known path convexities, corresponds to the case where $\mathcal{P}$ is the family of all shortest paths \cite{everett1985hull, buckley-1990, farber-1986}. In addition to this, there are other convexities like \emph{monophonic convexity} \cite{caceres-2005, source17, source20} and $P_3$-\emph{convexity}~\cite{source11, centeno2011irreversible, coelho2019p3} which are defined respectively over induced paths and paths with three vertices. There are also some  convexity notions that are not defined on the path systems; some notable examples are \emph{Steiner convexity}~\cite{source9} and $\Delta$-\emph{convexity}~\cite{bijo2, anand2025helly,bijo1}. For Steiner convexity, a set $S \subseteq V(G)$ is \emph{Steiner convex} if, for any subset $S' \subseteq S$, all vertices of any Steiner tree with terminal set $S'$ are contained within $S$.
 In $\Delta$-convexity, a set $S\subseteq V(G)$ is $\Delta$-\emph{convex} if every vertex $u\in V(G)\setminus S$ cannot form a triangle with any two vertices in $S$. Graph convexities have been widely explored in various contexts, including the computation of convexity invariants such as the hull number, the interval number, and the convexity number. A detailed analysis of geodesic convexity can be found in a major source work by Pelayo~\cite{pelayo-2015}.
Additionally, the recent book \cite{araujo2025introduction} provides a detailed survey on the computational aspects of graph convexity.
Moreover, graph convexities can eventually model contexts involving disseminating processes, such as the spread of information, opinions, or infections~\cite{dreyer2009irreversible}.

 We address a newly proposed graph convexity in this paper, called Cycle Convexity, by Araujo et al. in $\cite{interval08}$. In this convexity, given a graph $G$ and a subset of vertices $S \subseteq V(G)$, the interval function \textit{infects} $v\notin S$ if there exists a cycle in the same connected component of the induced subgraph of $S\cup \{v\}$, denoted by $G[S \cup \{v\}]$ containing $v$.
  For a set $S \subseteq V(G)$, the \textit{cycle interval} of $S$, denoted by $\intv(S)$, is the set formed by the vertices of $S$ and any $w \in V(G)$ that form a cycle with some vertices of $S$. If $\intv(S) = S$, then $S$ is \textit{cycle convex} in $G$. The \textit{cycle convex hull} of a set $S$, denoted by $\hullc (S)$, is the smallest cycle convex set containing $S$. We say that a vertex $w \in V(G)\setminus S $ is \textit{generated} by $S$, if $w\in \hullc(S)$. A vertex $w \in S$ is \textit{redundant} if $w \in \hullc (S \setminus \{w\}) $. If $\hullc (S)  = V(G)$, then $S$ is a \textit{cycle hull set}. 
  
  The study of cycle convexity in graphs has gained attention due to its applications in both graph theory and related areas, including Knot Theory~\cite{araujo2020cycle}. The concept of interval number and hull number in terms of  Cycle Convexity were introduced in $\cite{interval08}$ and $\cite{hull09}$. Regarding the interval number, the authors showed that deciding whether $\intv(G) \leq k$ is NP-complete for split graphs or a bounded degree planar graph and W[2]-hard for bipartite graph but can be computed in polynomial-time algorithm for outerplanar graphs, cobipartite graphs and interval graphs. They also presented an FPT algorithm for computing the interval number for $(q,q-4)$- graphs. Concerning the cycle hull number in~\cite{araujo2024hull}, they presented bounds for this parameter in $4$-regular planar graphs and proved that, given a planar graph $G$ and an integer $k$, determining whether $\hn(G) \leq k$ is NP-complete. They also showed that the parameter is computable in polynomial time for classes like chordal, $P_4$-sparse, and grid graphs. For the cycle convexity number, Lima, Marcilon, and Medeiros~\cite{lima2024complexity} proved that, given a graph $G$ and an integer $k$, determining whether $\con(G) \geq k$ is both NP-complete and W[1]-hard when parameterized by the size of the solution. However, for certain specific graph classes like extended $P_4$-laden graphs, they show an algorithm able to solve the problem in polynomial time. Also in $\cite{anand2025complexity}$, the authors present some properties of cycle convex sets in graph products, focusing on the hull number and the convexity number. Given a graph $G$ and an integer $k$, it is proved that determining whether $\hn(G) \leq k$ is NP-complete even if $G$ is a bipartite Cartesian product graph. For the cycle convexity number, they provide exact formulas for the Cartesian, strong, and lexicographic products of nontrivial connected graphs. In addition, the provided formulas directly imply the NP-completeness of the decision problem related to the convexity number for the three considered product graphs.

In this paper, we investigate the parameter exchange number under cycle convexity in graphs. 
From the computational complexity point of view, we prove that determining the decision version of the exchange number problem is NP-complete even for $K_5$-free graphs. 
On the positive side, for chordal graphs in which all blocks lie in a single chain, we derive a closed formula expressing the exchange number as a function of the number of blocks, which leads to a polynomial-time algorithm for computing the parameter. 
We also provide results for some graph classes such as trees, cycles, complete graphs, complete multipartite graphs, and graphs with universal vertices, showing that the exchange number is fixed on these classes. 
Additionally, we present a result for unicyclic graphs and characterized $n$-vertex graphs whose exchange number equals $n-1$, which also leads to polynomial-time computations. 
We establish further results on graph products, including a lower bound for the Cartesian product and determine exact values for the products of paths and complete graphs. Also, with ideas from the work by Anand et al.~\cite{bijo2}, we provide exact values for the strong and lexicographic products. 

The text is organized as follows. Section~\ref{sec:preliminaries} describes some fundamental concepts and terminologies. Section~\ref{sec:hardness} presents the hardness results. In Section~\ref{sec:classes}, we focus on the exchange number for some specific graph classes. Section~\ref{section-product} contains the results on graph products and Section~\ref{sec:conclusion} contains some further works.

\section{Preliminaries}
\label{sec:preliminaries}

All the graphs considered in this paper are connected, simple, and undirected. We denote the graph by $G=(V,E)$ and when needed, we write $V=V(G)$ and $E=E(G)$ to avoid ambiguity.
Given a graph $G$ and $u \in V(G)$, the \index{open neighbourhood} open and the  \index{closed neighbourhood} closed neighbourhood of a vertex $u$ are $N_{G}(u) = \{v : uv \in E(G)\}$ and $N_{G}[u] = N_{G}(v) \cup \{u\}$, respectively. A vertex $v$ in a connected graph $G$ is a \index{cut-vertex} \textit{cut-vertex} of $G$ if $G - v$ is disconnected. A \textit{component} of a given undirected graph $G$ may be defined as a connected subgraph that is not part of any larger connected subgraph. Given a graph $G=(V,E)$ and a vertex set $S \subseteq V(G)$,
we denote by $G[S]$ the \emph{subgraph of $G$ induced by $S$},
that is, the graph with vertex set $S$ and edge set
$E(G[S]) = \{ uv \in E(G) : u,v \in S \}$.

Given a graph $G$ and a set $S \subseteq V(G)$, the cycle $\textit{convex hull}$ of $S$, denoted by $\hullc (S)$, is the smallest convex set containing $S$. The  $\textit{hull number}$ of $G$ in the cycle convexity, or shortly, the cycle hull number of $G$, $\hn(G)$, is the cardinality of a smallest set $S$ such that $\hullc(S) = V(G)$. Let $S \subseteq V(G)$ be a nonempty finite set. The set $S$ is called $\textit{exchange dependent}$ $\textit{(or, $E$-dependent)}$, provided for each $p \in S, \hullc(S \setminus \{p\}) \subseteq \displaystyle\bigcup _{a \in S \setminus \{p\}}\hullc( S \setminus \{a\} )$, and it is called $\textit{exchange independent (or, $E$-independent)}$ otherwise. That is, $S$
is said to be $\textit{$E$-independent}$, if
$|S| = 1$ or there is a $p \in S$ such that $p' \in \hullc( S \setminus \{p\} ) \setminus\displaystyle \bigcup _{a \in S \setminus \{p\}} \hullc( S \setminus \{a\} )$. In this case, we say that $p$ is a \textit{pivot} and $p'$ is an \textit{anti-pivot} of $S$. The $\textit{exchange number},e_{cc}(G)$ is the maximum cardinality of an $\textit{$E$-independent}$ set of $G$.
A set of vertices $S$ of a graph $G$ is $\textit{Carathéodory dependent}$ $\textit{(or shortly, $C$-dependent)}$ provided  $\hullc(S)\subseteq \displaystyle\bigcup_{a \in S}\hullc(S \setminus \{a\})$  and it is $\textit{Carathéodory independent}$ $\textit{(or shortly, $C$-independent)}$ otherwise. That is, a subset $S$ of $V(G)$ is said to be $\textit{C-independent}$, if there is $p \in \hullc(S) \setminus\displaystyle \bigcup_{a \in S}\hullc(S \setminus \{a\})$.

       A graph $G$ is \textit{unicyclic}, if it is connected and contains precisely one cycle. A \textit{bicyclic graph} is a simple connected graph in which the number of edges equals the number of vertices plus one. 
       A path graph, denoted by $P_n$ is a graph with vertex set $\{v_1,v_2,\ldots,v_n\}$ and edges $\{v_i,v_{i+1}\}$ for $i\in \{1,2,\ldots,n-1\}$. A subset of vertices in a graph $G$ is \textit{independent} if the vertices are pairwise non-adjacent. The \textit{independence number} of a graph $G$, denoted by  $ \alpha (G)$, is the cardinality of the maximum independent set of vertices. A vertex $v \in V(G)$ is a $\textit{universal vertex}$ of $G$, if $v$ is adjacent to all other vertices of $G$. A connected graph $G$ is said to be $\textit{2-connected}$, if $G\setminus\{v\}$ is connected for every vertex $v \in V(G)$. A \textit{block} $B$ in a graph $G$ is a maximal $2$-connected subgraph of $G$. A block $B$ in a graph $G$ is said to be an \textit{end block}, if it contains only one cut-vertex. A \textit{chain of blocks} is a sequence of blocks in a graph $G$ such that each consecutive pairs of blocks shares a common cut vertex.

       Let $G$ and $H$ be two graphs. In Section~\ref{section-product}, we discuss the cycle exchange number of the \emph{Cartesian product} $G \Box H$, \emph{lexicographic product} $G \circ H$, and \emph{strong product} $G \boxtimes H$. The graphs $G$ and $H$ are called \textit{factors}. All these products have the vertex set $V(G) \times V(H)$ in common. For $(g_1, h_1), (g_2, h_2) \in V(G) \times V(H)$:
In the Cartesian product $G \square H$, the vertices $(g_1, h_1)$ and $(g_2, h_2)$ are adjacent if and only if either 
i) $g_1 \sim g_2$ in $G$ and $h_1 = h_2$, or 
ii) $g_1 = g_2$ and $h_1 \sim h_2$ in $H$.
In the lexicographic product $G \circ H$, these vertices are adjacent if either 
i) $g_1 \sim g_2$, or 
ii) $g_1 = g_2$ and $h_1 \sim h_2$.
Finally, in the strong product $G \boxtimes H$, the vertices $(g_1, h_1)$ and $(g_2, h_2)$ are adjacent if one of the following holds: 
i) $g_1 \sim g_2$ and $h_1 = h_2$, 
ii) $g_1 = g_2$ and $h_1 \sim h_2$, or 
iii) $g_1 \sim g_2$ and $h_1 \sim h_2$.
For a product $\ast \in \{\square, \circ, \boxtimes\}$,  $u \in V(G)$, and $v \in V(H)$, we define $^uH$ to be the subgraph of $G \ast H$ induced by $\{ u \} \times V(H)$, which we call an \emph{$H$-layer}, while the \emph{$G$-layer}, $G^v$ is the subgraph induced by $V(G) \times \{ v \}$. 

\section{Computing the Exchange Number in Cycle Convexity}
\label{sec:hardness}

This section deals with the computational complexity of determining the Exchange number of graphs in terms of Cycle Convexity. We begin this section by presenting some observations and properties of Exchange independent sets, some of which rely on concepts from Carathéodory independent sets. Lemma~\ref{lemma:e_ind_set_general} presents some properties of $E$-independent sets of a general graph convexity.

\begin{lemma}[\cite{anand2020caratheodory}]\label{lemma:e_ind_set_general}
Let $S$ be an $E$-independent set of a convexity $\mathcal{C}$. Then, the following statements hold:
\begin{enumerate}[label={\normalfont(\alph*)}]
    \item If $p$ is a redundant vertex of $S$, then $p$ is the only pivot of $S$.
    \item $S$ has at most one redundant vertex.
    \item If $S \setminus \{p\}$ is a hull set, then $p$ is the only pivot of $S$.
    \item $S \setminus \{u,v\}$ is not a hull set for any $u,v \in S$.
    \item If $p$ is a pivot of $S$, then $S \setminus \{p\}$ is $C$-independent.
\end{enumerate}
\end{lemma}


The following Lemma~\ref{lemma:c_barbosa} was proved for $P_3$-convexity~\cite{barbosa2012caratheodory}, and the same reasoning of the authors shows that it holds for cycle convexity too.

\begin{lemma}[\cite{barbosa2012caratheodory}]\label{lemma:c_barbosa}
    Let $S$ be a C-independent set of cycle convexity. Then
    \begin{enumerate}[label=(\alph*)]
    \item no proper subset $S'$ of $S$ satisfies $\hullc(S') = V(G)$.
    \item $\hullc(S)$ induces a connected subgraph of $G$.
    \end{enumerate}
\end{lemma}



The following results lead to the properties of $E$-independent sets stated in Lemma~\ref{lemma:e_ind_set_cc}. 
We begin with a basic proposition about the intersection of an arbitrary set of vertices $S$ with each component of its cycle convex hull.

\begin{proposition}\label{prop:connected_order_two}
Let $G$ be a graph, $S \subseteq V(G)$, and $H$ a connected component of $G[\hullc(S)]$. It holds that $|S \cap V(H)| = 1$ if $H$ is trivial and $|S \cap V(H)| \geq 2$ otherwise.
\end{proposition}

\begin{proof}
The statement follows by the definition of cycle convex hull, in which $\hullc(S\cap V(H)) = V(H)$, and by the definition of cycle interval, in which a vertex $v$ belongs to $\hullc(S\cap V(H))$ if and only if $v \in S \cap V(H)$ or $v$ lies in a cycle $vv_1, \dots, v_\ell$, with $\ell \geq 2$ and $v_1, \dots, v_\ell \in \hullc(S\cap V(H))$.
\end{proof}

\begin{lemma}\label{lemma:e_ind_set_cc}
If $S$ is an $E$-independent set of a graph $G$ in cycle convexity and $|S|\geq 3$, then the following holds
\begin{enumerate}[label={\normalfont(\alph*)}]    
    \item if $p$ is a pivot of $S$, then $G[\hullc(S \setminus \{p\})]$ is connected.
    \item $G[S]$ contains at least one edge.    
    \item $G[\hullc(S)]$ is either connected or it has exactly a trivial and a non-trivial connected component.
    \item no $\ell$ vertices of $S$ induce a $C_{\ell}$ in $G$.
    \item there is at most one vertex in $S$ that is not part of a cycle in $G$.
\end{enumerate}
\end{lemma}

\begin{proof}
Suppose that $S$ is an $E$-independent set with pivot $p \in S$ and anti-pivot $p' \in \hullc(S \setminus \{p\}) \setminus \bigcup _{a \in S \setminus \{p\}} \hullc(S \setminus \{a\})$.
\begin{enumerate}[label={\normalfont(\alph*)}]
    \item Given the pivot $p \in S$, Lemma~\ref{lemma:e_ind_set_general}(e) implies that $S \setminus \{p\}$ is a $C$-independent set, then Lemma~\ref{lemma:c_barbosa}(b) implies that $S \setminus \{p\}$ must induce a connected subgraph of $G$. 
    
    \item If $G[S]$ is edgeless, then  $\hullc(S \setminus \{p\}) = S \setminus \{p\}$. Since $|S| \geq 3$, $\hullc(S \setminus \{p\})$ induces an edgeless and disconnected subgraph of $G$, contradicting (a).
    
    \item By contradiction, suppose that $G[\hullc(S)]$ is disconnected and does not consist of exactly one trivial and one non-trivial connected component. 
    Denote the components of $G[\hullc(S)]$ by $G_1, \dots, G_\ell$. If $\ell = 2$, the contradiction hypothesis implies that $G_1$ and $G_2$ are non-trivial components. Then, by Proposition~\ref{prop:connected_order_two}, $|S \cap V(G_i)| \geq 2$, for $i = 1, 2$. Consequently $(S \setminus \{p\}) \cap V(G_i) \neq \emptyset$, for $i = 1,2$, and  $\hullc(S \setminus \{p\})$ induces a disconnected graph, a contradiction to (a).
    If $\ell \geq 3$, then regardless of which component contains $p$, there still remain at least two other components intersecting $S \setminus \{p\}$, so $\hullc(S \setminus \{p\})$ is disconnected, again a contradiction to (a).

    \item If $\ell$ vertices of $S$ induce a $C_\ell$, say $v_1, v_2, \dots, v_\ell$, then $\hullc(S \setminus \{v_i\} ) = \hullc(S)$ for every $1 \leq i \leq \ell$, contradicting the assumption that $S$ is an $E$-independent set. 
    \item Suppose, by contradiction, that there are two vertices $u, v \in S$ that do not lie on a cycle in $G$. Then $\hullc(S \setminus \{u\}) = \hullc(S) \setminus \{u\}$ and $\hullc(S \setminus \{v\}) = \hullc(S) \setminus \{v\}$.
    
    If the pivot $p \in \{u,v\}$, then $\hullc(S \setminus \{p\}) = \hullc(S) \setminus \{p\}$. Moreover,    
    $\bigcup_{a \in S \setminus \{p\}} \hullc(S \setminus \{a\}) = \hullc(S)$, then $\hullc(S) \setminus \{p\} \subseteq \hullc(S)$ contradicting that $S$ is an $E$-independent set.    
    Otherwise, if $p \in S \setminus \{u,v\}$, then $\bigcup_{a \in S \setminus \{p\}} \hullc(S \setminus \{a\}) = \hullc(S \setminus \{u\}) \cup \hullc(S \setminus \{v\}) = \hullc(S)$, again a contradiction. \qedhere
\end{enumerate}\end{proof}

Inspired in a construction by Barbosa et al.~\cite{barbosa2012caratheodory}, we show a hardness result for the decision problem related to the exchange number, defined in sequel.

\begin{problem}{\textsc{Exchange Number in Cycle Convexity}}\\
\textbf{Instance:} A graph $G$ and a positive integer $k$.\\
\textbf{Question:} Is $e_{cc}(G) \geq k$? 
\end{problem}

The proof uses the well-known \textsc{3-Satisfiability (3-SAT)} problem~\cite{garey1979computers} in which we are given a boolean formula 
$\phi = (X, \mathscr{C})$ in the conjunctive normal form (CNF), where $X$ is a set of variables 
and $\mathscr{C}$ is a collection of clauses such that each clause $c \in C$ contains exactly three literals. 
The question is whether there exists a truth assignment to the variables in $X$ that satisfies all clauses in $\mathscr{C}$.

\begin{theorem}\label{theo:NPc_ecc}
    \textsc{Exchange Number in Cycle Convexity} is NP-complete, even for $K_5$-free graphs.
\end{theorem}

\begin{proof}
    Computing the convex hull of a set of vertices in cycle convexity can be done in polynomial time~\cite{araujo2024hull}, then the NP-membership follows. Next, we present a polynomial-time reduction from the NP-complete problem \textsc{3-SAT}~\cite{garey1979computers} to \textsc{Exchange Number in Cycle Convexity}.

    From an instance $\phi = (X, \mathscr{C})$ of \textsc{3-SAT} with $|X| = n$ and $\mathscr{C} = \{C_1, \dots, C_m\}$, we construct an instance $(G,k)$ of \textsc{Exchange Number in Cycle Convexity} where $G$ is a graph as described next and $k = 2m+1$. We may assume that $m \geq 2$.

    For every  $i = 1, \dots, m$, we add to $G$ a clause gadget $G_{c_i}$ formed by a triangle $T_\ell = \{\ell_1, \ell_2, \ell_3\}$, for every literal $\ell \in C_i$, and a vertex $c_i$ adjacent adjacent to all vertices of every $T_\ell$. In addition, $G_{c_i}$ contains two adjacent vertices $w_i, w_i'$ and all the six possible edges between the vertices in $\{w_i, w_i'\}$ and $A_i = \{\ell_1 : \ell \in C_i \}$. Further, denote $A = A_1 \cup \dots \cup A_m$, $B_i = \{\ell_2, \ell_3 : \ell \in C_i\}$, and $B = B_1 \cup \dots \cup B_m$. See Figure~\ref{fig:construction_ecc} for a reference. Next, for every pair of opposite literals $\ell \in C_i$ and $\overline{\ell} \in C_j$, for $i,j \in \{1, \dots, m\}$, $i \neq j$, we add four vertices $\ell_{4}, \ell_{5}, \overline{\ell}_{4}, \overline{\ell}_{5}$ and the set of eleven edges 
    $\{c_i\ell_4, \ell_2\ell_4, \ell_3\ell_4, \ell_5\ell_4, \overline{\ell}_5\ell_4\} \cup 
    \{c_j\overline{\ell}_4, \overline{\ell}_2\overline{\ell}_4, \overline{\ell}_3\overline{\ell}_4, \ell_5\overline{\ell}_4, \overline{\ell}_5\overline{\ell}_4\} 
    \cup \{ \ell_5\overline{\ell}_5\}$. Put $\ell_4, \overline{\ell}_4 \in A$ and $\ell_5, \overline{\ell}_5 \in W$.
    Finally, add to $G$ the sets of vertices $D = \{d\}$, $Z = \{z, z'\}$, add the edges $\{ dc_1, dc_m \} \cup \{ c_ic_{i+1} : 1 \leq i < m\} \cup \{ zz' \}$, and make every vertex in $Z$ adjacent to every vertex in $W$.

    \begin{figure}[htb!]
    \centering
    \setlength{\fboxsep}{0pt} %
    \setlength{\fboxrule}{0pt}
    \fbox{
    \scalebox{0.99}{\input{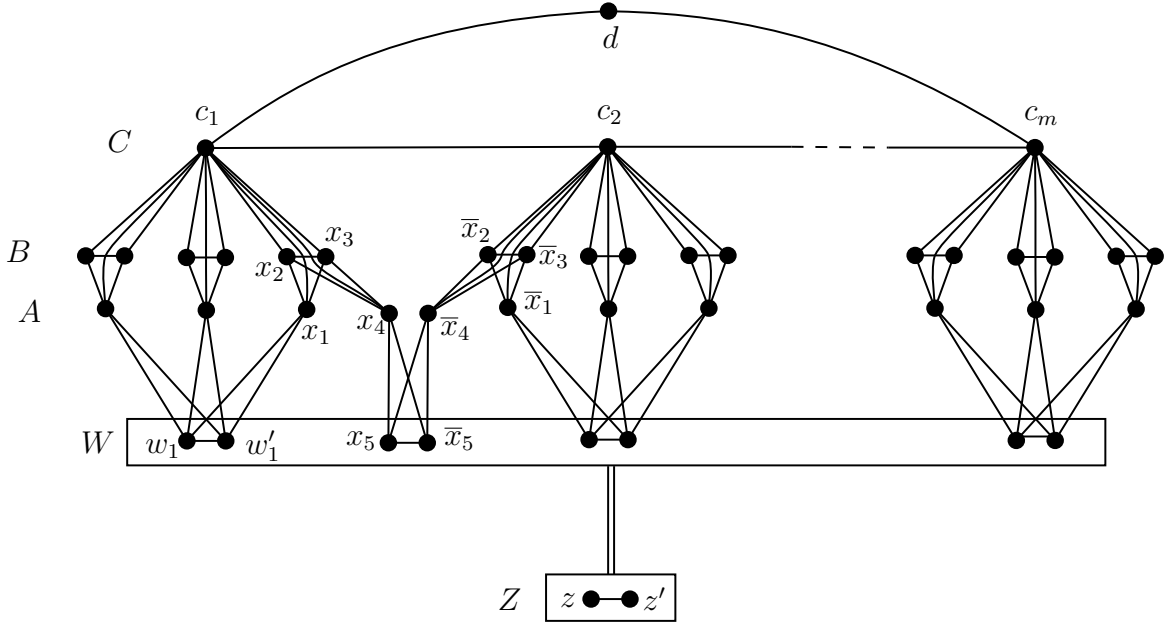}}
    }
    \caption{Sketch of the graph $G$ constructed for Theorem~\ref{theo:NPc_ecc} from an instance of 3-SAT, where $x \in C_1$ and $\overline{x} \in C_2$. The double edges represent all the possible edges between vertices in the two rectangles.
}
    \label{fig:construction_ecc}
\end{figure}

    The construction produces a graph $G$ with $12$ vertices for every clause gadget $G_{c_i}$, plus $3$ vertices from $D \cup Z$, plus $4$ vertices for every pair of opposite literals. Then $|V(G)| \leq 12m + 3 + 4\cdot O(m^2) = O(m^2)$ and the construction of $G$ can be accomplished in polynomial-time. It is easy to see that $G$ is a $K_5$-free graph.

    Next, we show that $\phi$ is satisfiable if and only if $G$ has an Exchange set of order at least $k = 2m+1$.

    Suppose that $t: \phi \to \{T,F\}$ is a truth assignment that satisfies all clauses in $\mathscr{C}$. We show how to find an $E$-independent set $S$ of $G$ with $|S| \geq 2m+1$. First, we put $z \in S$. Next, for every clause $C_i \in \mathscr{C}$, for $1 \leq i \leq m$, we choose one and only one literal $\ell \in C_i$ which $t(\ell) = T$ and let $\ell_2, \ell_3 \in S$. This makes $|S| = 2m+1$, and we claim that $z$ is a pivot of $S$. Let us analyze $\hullc(S \setminus \{z\})$. For the chosen literal $\ell \in C_i$, we have that $\ell_1$ lies in the triangle $T_\ell = \{\ell_1,\ell_2,\ell_3\}$, as well as $c_i$ lies in the cycle $c_i,\ell_2,\ell_3$, and if there exists $\overline{\ell}$ an opposite literal of $\ell$, then $\ell_4$ exists and $\ell_4$ belong to a cycle $\ell_2,\ell_3, \ell_4$. All of this imply that $\ell_1, \dots, \ell_4, c_i \in \hullc(S \setminus \{z\})$. Further, given that all the clauses are satisfied, we get that $c_1, \dots, c_m \in \hullc(S\setminus \{z\})$. In consequence, $d$ lies in a cycle induced by $C \cup D$, hence $d \in \hullc(S\setminus \{z\})$.
    This makes $\hullc(S \setminus \{z\}) =  H = C \cup D \cup \{\ell_1, \ell_2, \ell_3, \ell_4 : \ell \in C_i \text{ such that } \ell \text{ is exactly one true literal chosen in } C_i \}$, since the construction implies that every vertex in $V(G) \setminus H$ has at most one neighbor in $H$.
    
    Next, let $a \in S\setminus \{z\}$ and we evaluate $\hullc(S \setminus \{a\})$. Since the case $a = \ell_3$ is symmetric, let us fix $a = \ell_2 \in V(G)$ corresponding to a chosen literal $\ell \in C_i$, for some $i \in \{1, \dots, m\}$. 
    It is clear that $c_i \notin  \hullc(\{\ell_3\}) = \{\ell_3\}$, and for every $j \in \{1, \dots, m\}$, $j \neq i$, $c_j \in \hullc(S \setminus \{a\})$. Moreover, given that no neighbor of $z$ belongs to $\hullc(S \setminus \{a\})$, we get that $z$ is an isolated vertex in $\hullc(S \setminus \{a\} )$. Then, by construction, $d, c_i \notin \hullc(S \setminus \{a\})$. This implies that $d \in \hullc(S \setminus \{z\}) \setminus  \bigcup_{a \in S\setminus \{z\}}\hullc(S \setminus \{a\}) $ and $S$ is an $E$-independent set.

    For the converse, let $S$ be an Exchange set with $|S| \geq 2m+1$. The proof goes through three claims.

    \medskip
    \noindent \textbf{Claim~1.} Let $S' = S \cap (A \cup W \cup Z)$. Then $S'$ induces an edgeless subgraph of $G$.

    \smallskip
    \noindent Suppose, by contradiction, that $uv \in E(G[S'])$. We focus on a main case.

    \smallskip
    \noindent \textbf{Case 1.} $\{u,v\} = \{z,z'\}$. Since $z,z'$ are adjacent to every $w \in W$, we get that $W \subseteq I_{cc}(S')$. Next, the construction implies that every $a \in A$ has two adjacent neighbors in $W$, then $A \subseteq I_{cc}(S')$. Further, for every $i = 1, \dots, m$, there is a cycle  $c_i\ell_1w_i\ell'_1$ where $\ell, \ell'$ are literals in $C_i$, then $c_i \in I_{cc}(A \cup W)$, which gives $C \subseteq I_{cc}(A \cup W)$. Consequently, for every literal $\ell \in C_i$, we have that $\ell_2, \ell_3$ are neighbors of both $\ell_1$, and $c_i$, then $B \subseteq I_{cc}(A \cup C)$. Finally, $d$ lies in a cycle $d,c_1, \dots, c_m$, then $d \in \hullc(S')$ which means that $\hullc(S') = V(G)$. Since that $|S| \geq 2m \geq 4$, there exist $s,s' \in S$ such that $\{z,z'\} \subseteq S \setminus \{s,s'\}$, contradicting Lemma~\ref{lemma:e_ind_set_general}(d).
    
    \smallskip
    \noindent In any other case, the construction implies that two adjacent vertices $w, w' \in \hullc(S')$, which implies that $z, z' \in \hullc(S')$, and as in Case~1 we find $\hullc(S') = V(G)$, giving a contradiction. 

    \medskip
    \noindent \textbf{Claim~2.}  Let $Q_i= \{ \ell_1, \ell_2, \ell_3, \ell_4: \ell \in C_i \} \cup \{c_i\}$, for $1 \leq i \leq m$. Then $1 \leq |S \cap Q_i| \leq 2$.

    \smallskip
    \noindent We prove the upper bound first. By contradiction, let $v_1,v_2,v_3 \in S \cap Q_i$. By Lemma~\ref{lemma:e_ind_set_cc}(d) we know that $\{v_1,v_2, v_3\}$ does not induce a $C_3$ in $G$. 
    If $\{v_1,v_2, v_3\}$ induce a $K_2$, say $v_1v_2 \in E(G)$, by construction we have that $\{v_1,v_2\} \cap B \neq \emptyset$. So, we assume w.l.o.g. that $v_1 = \ell_2$ for some literal $\ell \in C_i$. We have that $V(T_\ell) \cup \{ \ell_4 \} \cup \{c_i\} \subseteq \hullc(S)$. Then, if $v_3 \in V(T_{\ell'}) \cup \{\ell'_4 \}$, for some literal $\ell' \in C_i$, $ \ell \neq \ell'$, since $c_i \in \hullc(S)$, we have that $V(T_{\ell'}) \cup \{\ell'_4 \} \subseteq \hullc(S)$. Consequently $w_i, w_i' \in \hullc(\{\ell_1, c_i, \ell'_1\})$ and we reach a contradiction as in Claim~1. 
    
    Now, if $\{v_1,v_2, v_3\}$ induce an edgeless graph, since $G[\hullc(S)]$ is either connected or it has exactly a trivial and a non-trivial connected component (Lemma~\ref{lemma:e_ind_set_cc}(c)), there exist $v \in N_G(S'') \cap \hullc(S)$, where $S''$ is a subset of $ \{v_1, v_2, v_3\}$ with at least two elements. If $v = c_i$, this implies that $V(T_{\ell}) \cup V(T_{\ell'}) \cup \{\ell_4, \ell'_4 \} \subseteq \hullc(S)$, for distinct literals $\ell$ and $\ell'$ associated with the intersection of $S''$ with $ V(T_\ell) \cup V(T_{\ell'})$. Thus $w_i, w_i' \in \hullc(\{\ell_1, c_i, \ell'_1\})$ a contradiction. If $v \neq c_i$, then $v \in \{w_i,w_i'\}$ and by construction $w_i,w_i' \in \hullc(S)$ a contradiction.
    Hence, $|S \cap Q_i| \leq 2$.
    
    Up to this point, recall that $|S| \geq 2m+1$. By Claim~1, the set $S \cap (A \cup W \cup Z)$ induces an edgeless subgraph, and by Lemma~\ref{lemma:e_ind_set_cc}(c), $G[\hullc(S)]$ is either connected or consists of exactly one trivial and one non-trivial component. This implies that $|S \cap (W \cup Z)| \leq 1$. Given also that $|S \cap D| \leq 1$, we get that $|S \cap (A \cup B \cup C)| = |S \setminus (D \cup W \cup Z)| \geq 2m-1$. Since  $|S \cap Q_j| \leq 2$ for every $1 \leq j \leq m$, the the lower bound $|S \cap Q_i| \geq 1$ holds.

    \medskip
    \noindent \textbf{Claim~3.} Let $R_i = \{ \ell_1, \ell_2, \ell_3, \ell_4: \ell \in C_i \}$, for $1 \leq i \leq m$. Then $\hullc(S) \cap R_i \neq \emptyset$.

    \smallskip
    \noindent Suppose by contradiction that there is $i \in \{ 1, \dots, m\}$ such that $\hullc(S) \cap R_i = \emptyset$. By Claim~2, we have that $|S \cap Q_i| \geq 1$, since $Q_i \setminus R_i = \{c_i\}$, then $c_i \in S$. Claim~1 and Lemma~\ref{lemma:e_ind_set_cc}(c) imply that $|S \cap (W \cup Z)| \leq 1$ and by Claim~2, $|S \cap Q_j| \leq 2$, for every $1 \leq j \leq m$. Thus $S \cap D \neq \emptyset$, $S \cap (W \cup Z) \neq \emptyset$, as well as $|S \cap Q_j| = 2$, for every $1 \leq j \leq m$, $j \neq i$. Since $G[\hullc(S)]$ is either connected or consists of exactly one trivial and one non-trivial component, and $S \cap (W \cup Z)$ is an isolated vertex in $G[\hullc(S)]$, then $\hullc(S \setminus Z)$ induces a connected graph.

    Let $j \in \{1, \dots, m\} \setminus \{i\}$ and $S \cap R_j = \{v_1,v_2\}$. If $v_1v_2 \notin E(G)$, we get a contradiction since $\hullc(S \setminus Z)$ is connected. 
    So, assume that $v_1v_2 \in E(G)$.   
    We get that $c_j \in \hullc(S)$, also $\{\ell_1, \ell_2, \ell_3, \ell_4: \ell \in C_j \text{ is a literal associated to } v_1 \text{ or } v_2 \} \subseteq \hullc(\{v_1,v_2\})$. This implies that no vertex in every $Q_j$ is an anti-pivot, a contradiction. 

    \medskip
    By Claim~3, we know that $\hullc(S) \cap R_i \neq \emptyset$ for all $1 \leq i \leq m$. Moreover, the construction implies that a vertex $v$ belongs to $\hullc(S) \cap Q_i$ if and only if $S \cap Q_i \neq \emptyset$. Since Claim~2 ensures that $|S \cap Q_i| \leq 2$ and $G[S \cap Q_i] \simeq K_2$, we can design a truth assignment $t : \phi \to \{T,F\}$ by defining, for each $\ell \in C_i$, $t(\ell)=T$ if $S \cap Q_i \cap T_\ell \neq \emptyset$, and $t(\ell)=F$ otherwise. Finally, because $\hullc(S) \cap R_i \neq \emptyset$ for every $1\leq i \leq m$, this assignment $t$ satisfies all clauses in $\mathscr{C}$.
\end{proof}

In view of the computational hardness of computing the exchange number, we focus on restricted graph classes in the following sections.

\section{Exchange Number in Graphs} \label{sec:classes}
Let us begin this section by presenting the exchange numbers of some well-known graph classes. 

\begin{remark}\label{exchangecycle}
Let $G$ be a graph. It holds that: 
\begin{enumerate}
    
 \item[(1)] If $G$ is a cycle $C_n$; $n \geq 3,$ then $ e_{cc}(G)=2.$
  
     \item[(2)] If $G$ is a tree, then $e_{cc}(G)=2$.

     \item[(3)] If $G$ is a complete graph $K_n$; $n\geq 3,$ then $ e_{cc}(G)=2.$

      \item[(4)] If $G$ is a complete multipartite graph $K_{m_1,m_2,\ldots,m_n}$; $n \geq 2$,  then $ e_{cc}(G)=2.$
      
      \item[(5)] If $G$ contains at least 2 universal vertices, then $e_{cc}(G)=2.$
     \end{enumerate}
\end{remark}   

The following remark is an immediate consequence of the definition of $E$-independent sets.

\begin{remark}\label{E-dependentexchage}
    Let $S \subseteq V(G)$ be such that $S$ is the set of all independent vertices of $G$. Then $S$ is $E$-dependent in $G$.
\end{remark}

We now turn our attention to unicyclic graphs.

\begin{lemma}\label{exchangeunicyclegraph}
            If $G$ is a unicyclic graph of order $n$ containing a  cycle $C_m$ with $ 3 \leq m <n$, then $e_{cc}(G)=m.$
        \end{lemma}
        
        \begin{proof}
            Let $G$ be a unicyclic graph of order $n$. Then $G$ has a cycle $C_m$ $(3 \leq m < n)$ with vertex set $\{a_1, a_2, \ldots, a_m\}$. Let $S \subseteq V(G)$ with $|S|=m$. Consider the set  $S=\{a_1,a_2,\dots,a_{m-1},b\}$ where, $b \in V(G) \setminus V(C_m)$. Then, $\hullc(S \setminus \{b\})=\{a_1,a_2,\dots,a_{m-1},a_m\}, \hullc(S \setminus\{a_1\})=\{a_2,a_3,\dots,a_{m-1},b\},\hullc(S \setminus \{a_2\})=\{a_1,a_3,\dots,a_{m-1},b\},\dots, \hullc(S \setminus \{a_{m-1}\})=\{a_1,a_2,a_3,$ $\dots, a_{m-2},b\}$. Then, $a_m \in \hullc(S\setminus \{b\})$ but $a_m \notin \displaystyle \bigcup_{a \in S\setminus\{b\}}{\hullc(S \setminus \{a\})}$. Thus $S$ is an  $E$-independent set of $G$ with $|S|=m$. Now, it remains to show that there does not exist an $E$-independent set of size greater than $m$. Assume, there exists an $E$-independent set $S'\subseteq V(G)$ with $|S'|= m+1$. If $S'$ is the set of all independent vertices of $G$, then $S'$ is $E$-dependent by  Remark \ref{E-dependentexchage}. Otherwise, let us assume that $S'$ contains $m$ vertices from $C_m$. Consider $S'=\{a_1,a_2,\dots,a_{m},b\}$. Then, $\hullc(S'\setminus\{a_1\})=\hullc(S'\setminus\{a_2\})=\dots=\hullc(S'\setminus\{a_m\})=\hullc(S')$, which is impossible. Assume that $S'$ contains $i$ vertices from $C_m$, where $i\leq m-1$. For $i=m-1$, let $S'=\{a_1,a_2,\ldots,a_{m-1},b_1,b_2\}$, where $b_1,b_2\in V(G)\setminus V(C_m)$. Then, 
            $a_m \in \hullc(S'\setminus \{b_1\})$  as well as $a_m \in \hullc(S'\setminus\{b_2\})$, a contradiction. Now, let $i<m-1$. Then, no three vertices $\{a,b,c\}\subseteq V(G)\cap S'$ form a triangle in $G$ and hence $\hullc(S'\setminus \{a\}) \subseteq \hullc(S'\setminus \{b\})\cup\hullc(S'\setminus\{c\})$. 
              Therefore, the  maximum $E$-independent set of $G$ contains only $m$ vertices and the exchange number, $e_{cc}(G)=m$.         \end{proof}

For characterizing $n$-vertex graphs with exchange number $n-1$, we consider the class of graphs denoted by $C_{(n-1),1}$ and is defined as the family of connected unicyclic graphs that contain a cycle of order $n-1$ with a pendent vertex attached.

\begin{theorem}\label{theo:carac_n-1}
           
A connected graph $G$ 
of order $n\geq 3$ 
has exchange number, $e_{cc}(G)=n-1$ if and only if $G \in \{K_3,P_3,C_{(n-1),1}\}$.
    \end{theorem}
\begin{proof}

 For $G=P_3, K_3$, we have $e_{cc}(P_3)=e_{cc}(K_3)=2=n-1$ (Remark~\ref{exchangecycle}). Consider the unicyclic graph $G=C_{(n-1),1}$. Let $\{a_1,a_2,\dots,a_{n-1}\}$ be the vertices of the cycle in $G$ and let $a_n$ be the pendent vertex. Then, from Lemma \ref{exchangeunicyclegraph}, $S=\{a_1,a_2,\ldots,a_{n-2},a_n\}$ is an $E$-independent set and $e_{cc}(G)=n-1$. \\
 Conversely assume that $e_{cc}(G)=n-1$. Suppose $G\notin\{K_3,P_3,C_{(n-1),1}\}$. If $G=K_n;  n\geq 4$, then every pair of adjacent vertices forms an $E$-independent set of $G$ and $e_{cc}(G)=2\neq n-1$.
  Suppose $G$ is an acyclic graph with  $G\neq P_3$. Then, from the Remark \ref{exchangecycle}, $e_{cc}(G) \neq n-1$. Suppose $G \neq C_{(n-1),1}$. If $G=C_n; n\geq 4,$ then $e_{cc}(G)=2$, which is not possible. Assume that $G$ contains a cycle $C_m; 3 \leq m<n$. Let  $\{a_1,a_2,\ldots,a_m\}$ be the vertices of the cycle in $G$ and $\{a_{m+1},a_{m+2},\ldots,a_n\}\subseteq V(G)\setminus V(C_m)$. Then, by Lemma \ref{exchangeunicyclegraph}, $e_{cc}(G)=m \neq n-1$. 
  Suppose $G$ contains two or more distinct cycles. Let $C_1$ and $C_2$ be the two cycles in $G$ with $|V(C_1)|=p$ and $|V(C_2)|=q$. Consider a subset $S\subseteq V(G)$ with $|S|=n-1$ and let $u\in V(G)\setminus S$.
  If $u\in S'\cap V(C_1)$, 
  then $ u\in\displaystyle\bigcup_{a\in S'\cap V(C_2)}\hullc(S'\setminus\{a\})$. Similarly, if $u\in S'\cap V(C_2)$, then $u\in\displaystyle\bigcup_{a'\in S'\cap V(C_1)}\hullc(S'\setminus\{a'\})$ and thus $S'$ is $E$-dependent. 
Hence, we conclude that $G\in \{K_3,P_3,C_{(n-1),1}\}$.     
\end{proof}

We now focus on the subclasses of chordal graphs and adopt the following notations. 

\begin{notation}\label{not:block}
Let $G$ be a chordal graph. Then
\begin{itemize}
    \item[(1)] $\mathscr{B}(G)=\{B:B$ is a block of $G$ $\}$.
    
    \item[(2)] $C(G)=\{c:c$ is a cut vertex of $G$ $\}$.
    
    \item[(3)] For $B \in \mathscr{B}(G)$ and $S\subseteq V(G)$, denote $G_{B}=\{ u \in V(B) : \hullc(S\setminus \{u\})= \hullc(S) \setminus \{u\}\}$.
\end{itemize}
\end{notation}

\begin{notation}\label{not:single_chain}
Let $G$ be a chordal graph whose blocks are arranged in a single chain. Then
\begin{enumerate}
    
\item[(a)] Let $B \in \mathscr{B}(G)$ is not an end block of $G$, and let $\{u,v\}\subseteq V(B)$, where $u,v$ are not cut vertices of $G$. Consider the set $S=\{u,v,u_1,u_1',u_2,u_2',\ldots,u_k,u_k'\}\subseteq V(G)$, where $u_i\in V(B_i), u_i'\in V(B_i')$ for $i\in \{1,2, \ldots,k\}$ such that   $N_G(V(B))\cap S=\emptyset$. Moreover, $B_1,B_2,\ldots,B_k$ being the blocks lying to the right of $B$ and $B_1',B_2',\ldots,B_k'$ are those blocks lying to the left of $B$. Then, 
 $\Gamma'_{B}(u,v)= 
      \begin{cases}
     \{u,v\}; &\textit{if $uv \notin E(G)$,}\\ 
V(B); &\textit{if $uv\in E(G)$.}
\end{cases}$

\begin{figure}[H]
   \centering
\includegraphics[width=100mm,scale=1]{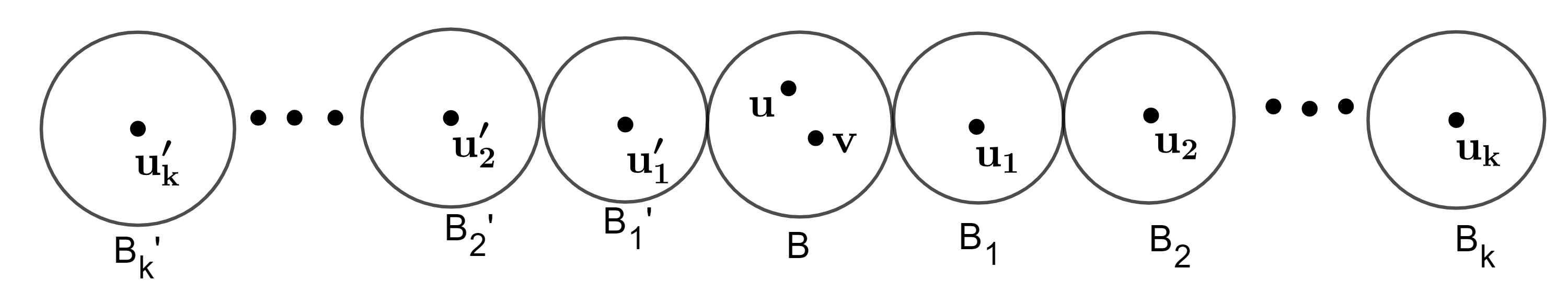}

    \caption{ A chain of blocks in which $uv \notin E(G), u_i\in V(B_i), u_i'\in V(B_i')$ for $1\leq i \leq k$, $ N_G(V(B))\cap S=\emptyset$ and $\Gamma'_B{(u,v)}=\hullc(u,v)=\{u,v\}.$}
    \label{fig 2}
  \centering
   \end{figure}

\begin{figure}[H]
    \centering
\includegraphics[width=100mm,scale=1]{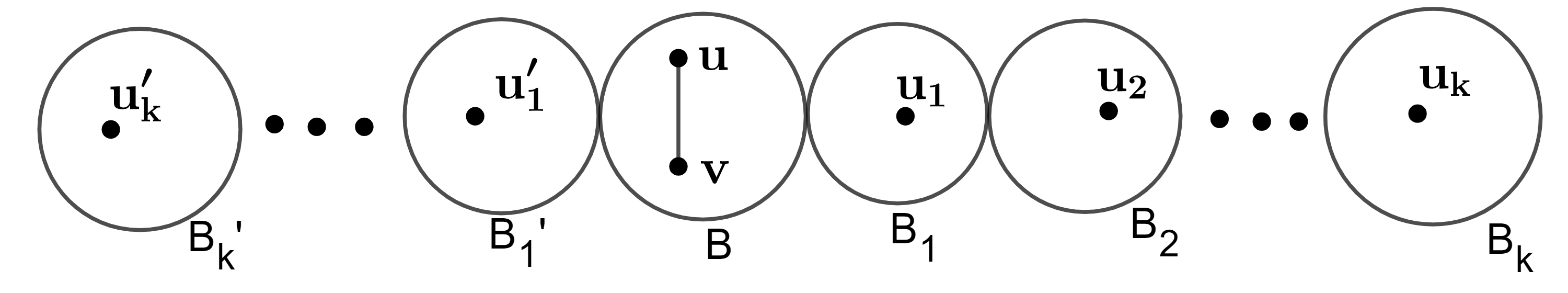}

    \caption{ A chain of blocks in which $uv \in E(G), u_i\in V(B_i), u_i'\in V(B_i')$ for $1\leq i \leq k$,   $ N_G(V(B))\cap S=\emptyset$ and $\Gamma'_B{(u,v)}=\hullc(u,v)= V(B).$}
    \label{fig 3}
  \centering
   \end{figure}

   \item[(b)] Let $B \in \mathscr{B}(G)$ is not an end block of $G$ and $\{u,v\}\subseteq V(B)$ are the cut vertices of $G$. Consider the set $S=\{u,v,u_1,u_1',u_2,u_2',\ldots,u_k,u_k'\}\subseteq V(G)$ and  $N_G(V(B))\cap S \neq \emptyset$. Let $u_1\in N_G(V(B))\cap V(B_1), u_2\in N_G(V(B_1))\cap V(B_2),\ldots,u_k\in N_G(V(B_{k-1}))\cap V(B_k)$. Similarly, $u_1'\in N_G(V(B))\cap V(B_1'), u_2'\in N_G(V(B_1'))\cap V(B_2'),\ldots, u_k'\in N_G(V(B_{k-1}'))\cap V(B_k')$, where, the blocks  $B_1,B_2,\ldots,B_k$ are lying to the right of $B$ and $B_1',B_2',\ldots,B_k'$ are those blocks lying to the left of $B$ with $\hullc(v,u_1,u_2,\dots,u_k)= V(B_1)\cup V(B_2)\cup\dots\cup V(B_k)$ and $\hullc(u,u_1',u_2',\dots,u_k')=V(B_1')\cup V(B_2')\cup\dots\cup V(B_k')$. Then,\\

 $\Gamma''_{B}(u,v)= 
      \begin{cases}
\displaystyle \bigcup_{i=1}^k V(B_i)\cup V(B_i'); &\textit{if $uv \notin E(G)$,}\\
V(B)\cup\displaystyle \bigcup_{i=1}^k V(B_i)\cup V(B_i'); &\textit{if $uv\in E(G)$.}

\end{cases}$
\end{enumerate}

    \begin{figure}[H]
    \centering
\includegraphics[width=100mm,scale=1]{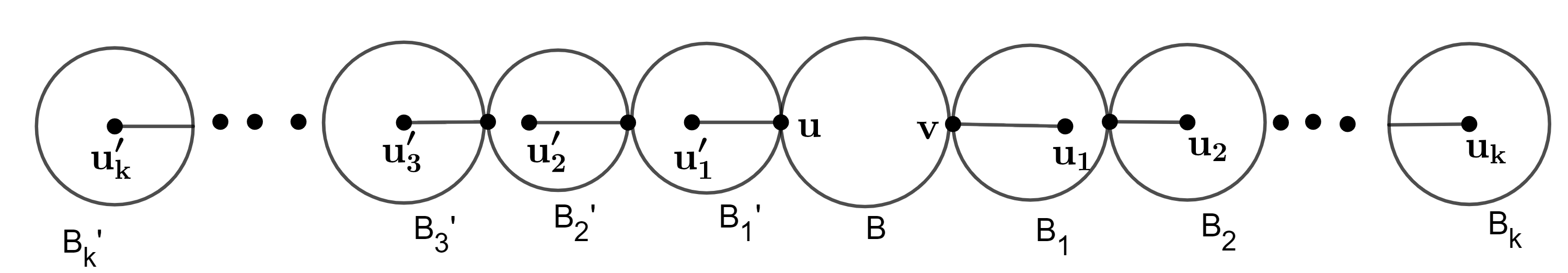}

    \caption{A chain of blocks where $uv \notin E(G), u_1\in N_G(v)\cap V(B_1),u_2 \in N_G(V(B_1))\cap V(B_2), \dots, u_k \in N_G(V(B_{k-1}))\cap V(B_k), u_1'\in N_G(u)\cap V(B_1'), u_2'\in N_G(V(B_1'))\cap V(B_2'), \dots, u_k'\in N_G(V(B'_{k-1}))\cap V(B_k')$ with  $\Gamma''_{B}{(u,v)}=\displaystyle \bigcup_{i=1}^k V(B_i)\cup V(B_i')$. }
   \label{tol1}
  \centering
   \end{figure} 
    
    \begin{figure}[H]
    \centering
\includegraphics[width=100mm,scale=1]{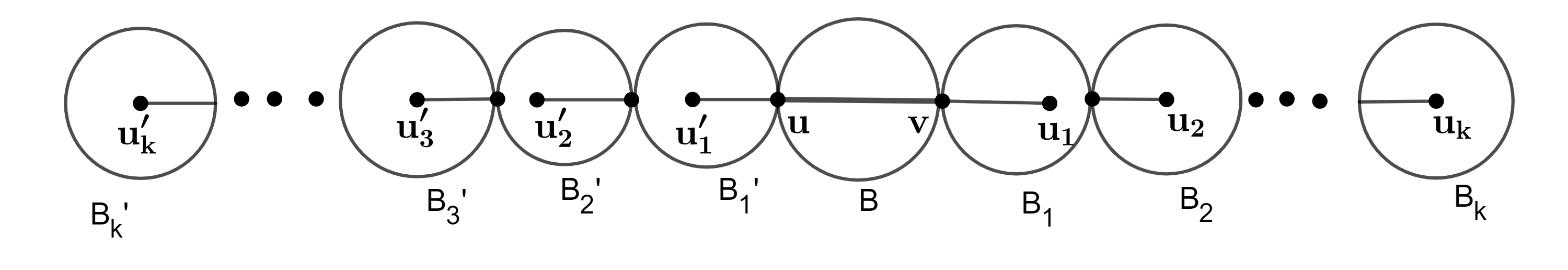}

    \caption{ A chain of blocks where $uv \in E(G), u_1\in N_G(v)\cap V(B_1),u_2 \in N_G(V(B_1))\cap V(B_2), \dots, u_k \in N_G(V(B_{k-1}))\cap V(B_k), u_1'\in N_G(u)\cap V(B_1'), u_2'\in N_G(V(B_1'))\cap V(B_2'), \dots, u_k'\in N_G(V(B'_{k-1}))\cap V(B_k')$ with $\Gamma''_{B}{(u,v)}=V(B)\cup\displaystyle \bigcup_{i=1}^k V(B_i)\cup V(B_i')$.}
  \centering
   \end{figure} 
\end{notation}

We use the notation $\Gamma_B(u,v)$ to denote  $\Gamma'_B(u,v)\bigcup \Gamma''_B(u,v)$ in the proof of the Theorem \ref{chordalexchangenumber}. \\

Before proceeding to our next results, we need two more definitions.
    
\begin{definition}[\cite{anand2025caratheodory}] \label{def:edge_vertex_property}   
    A block $B$ is said to have the \textit{edge-vertex property}, if there exists an edge $uv$ in $E(B)$ and a vertex $x \in V(B)$, such that $d(u,x)\geq 2$ and $d(v,x)\geq 2$.
\end{definition}
       
\begin{definition} \label{def:vertex_separation_property} 
    A block $B$ in a connected graph $G$ is said to satisfy the vertex-separation property, if there exist two vertices $x,y \in V(B)$ such that $x \in  N(c); c\in C(G)$ and $y \notin C(G) \cup N(c) \cup N(x)$.
 \end{definition}

    \begin{figure}[H]
    \centering
\includegraphics[width=40mm,scale=1]{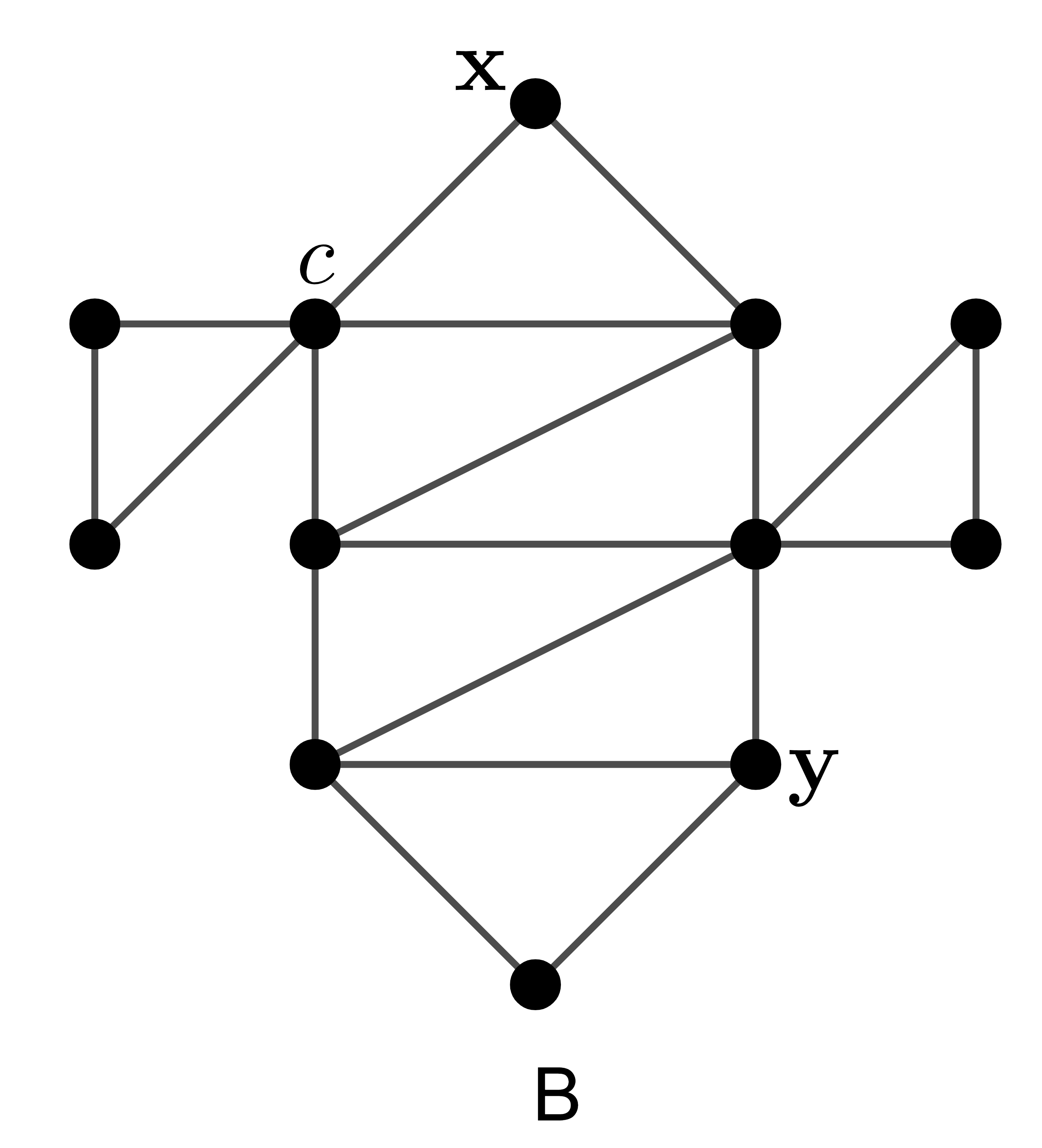}

    \caption{ An example graph where $x,y \in V(B)$ satisfy the vertex-separation property.}
    \label{fig 6}
  \centering
   \end{figure} 


\begin{theorem}\label{chordalexchangenumber}
    Let $G$ be a chordal graph with $l$ blocks where all blocks are arranged in a single chain, and no block is isomorphic to the complete graph $K_2$. Then the following statements hold.
    
    \begin{enumerate}
        \item[(i)] If no block of $G$ satisfies the vertex-separation property or the edge-vertex property, then $e_{cc}(G)=l+1$.
        
        \item[(ii)] If at least one  block of $G$ satisfies the vertex-separation property, or if one of the end blocks satisfies the edge-vertex property, then $e_{cc}(G)=l+2$.
    \end{enumerate}
\end{theorem}

\begin{proof}

\begin{enumerate}
   
\item[(i)] Consider a chordal graph $G$ with $l$ blocks and no block is isomorphic to the complete graph $K_2$. Assume that all the blocks are lying on a single chain. Let $B_1,B_2,\dots,B_l$ be the blocks of $G$ and let $S\subseteq V(G)$ be an $E$-independent set of $G$. Assume that $G$ contains no block satisfying the vertex-separation property or the edge-vertex property. Let $S=\{u_1,u_1',u_2,u_3,\dots,u_l\}$, where $\{u_1,u_1'\} \subseteq V(B_1)$ with $u_1u_1' \in E(G)$ and $u_i \in N(V(B_{i-1}))\cap V(B_i) $ for $i\in \{2,3,\dots,l\}$. Then, we have $\hullc(S\setminus \{u_1\})= S \setminus \{u_1\}$,$\hullc(S\setminus \{u_1'\})= S\setminus \{u_1'\},\displaystyle\bigcup_{i=2}^{l}\hullc(S \setminus \{u_i\})=\displaystyle\bigcup_{i=1}^{l-1}V(B_i)\cup \{u_l\}$ and $\hullc(S\setminus\{u_l\})=\displaystyle\bigcup_{i=1}^{l-1}V(B_i)$. 
 Therefore, 
$\displaystyle\bigcup_{i=1}^{l}\hullc(S \setminus \{u_i\})\not\subseteq\hullc(S\setminus\{u_l\})$ implies that $S$ is an $E$-independent set containing $l+1$ vertices.\\
 
Finally, we need to prove that there cannot exist an $E$-independent set with more than $l+1$ vertices. Suppose $G$ contains an $E$-independent set $S'$ with $|S'|\geq l+2$. Let us consider $|S'|=l+2$. Now examine the following cases.

\textbf{Case 1:} $S'$ contains at least one vertex from each block of $G$.
\begin{itemize}
    \item[(a)] Suppose $ S'$ contains two vertices from two distinct blocks $B_i, B_j$ of $G$, where, $i,j \in \{1,2,\dots,l\}$. Assume that  $x,x',y,y' \in S'$ with $ \{x,x'\}\subseteq V(B_i)$ and $\{y,y'\}\subseteq V(B_j)$. We need to prove that $\hullc(S'\setminus\{x\})\cup\hullc(S'\setminus\{y\})=\hullc(S')$.
Let $v\in \hullc(S')$. To prove $ v \in\hullc(S'\setminus \{x\})\cup\hullc(S'\setminus \{y\})$, consider the following subcases.

\textbf{Subcase 1:}  $v \in \hullc(\Gamma_{B_i}{(x,x')} \cup\Gamma_{B_j}{(y,y')})$.

 Assume that  $\hullc(\Gamma_{B_i}{(x,x')} \cup \Gamma_{B_j}{(y,y')})\neq\Gamma_{B_i}{(x,x')} \cup \Gamma_{B_j}{(y,y')}$. For some vertex  $v\in V(B_k)$ with $ 1<k<l$, let $v \in \hullc(\Gamma_{B_i}{(x,x')} \cup \Gamma_{B_j}{(y,y')}) \setminus(\Gamma_{B_i}{(x,x')} \cup \Gamma_{B_j}{(y,y')}) $. This implies that there exists vertices $z \in \Gamma_{B_i}{(x,x')}$ and $ z'\in  \Gamma_{B_j}{(y,y')}$ with $zz'\in E(G)$ such that $v \in \hullc(z,z')$. Therefore, the vertices $\{z,z',v\}$ satisfy the edge-vertex property, which is a contradiction to our assumption that $G$ cannot contain blocks having the edge-vertex property. Therefore, $\hullc(\Gamma_{B_i}{(x,x')} \cup \Gamma_{B_j}{(y,y')})=\Gamma_{B_i}{(x,x')} \cup \Gamma_{B_j}{(y,y')}$. Then, for any $v \in \Gamma_{B_i}{(x,x')}$, we get $v \in \hullc(S'\setminus\{y\})$ and for any $v \in \Gamma_{B_j}{(y,y')}$, we get $v \in \hullc(S'\setminus\{x\})$. Thus, $v \in \hullc(S'\setminus \{x\})\cup\hullc(S'\setminus \{y\})$. 
\vspace{0.2cm}

\textbf{Subcase 2:} $v \in \hullc(S')\setminus \hullc(\Gamma_{B_i}{(x,x')} \cup\Gamma_{B_j}{(y,y')})$. Then, $v \in \hullc(S'\setminus \{x\})$ or $v \in \hullc(S'\setminus \{y\})$ implies that $v \in \hullc(S'\setminus \{x\})\cup\hullc(S'\setminus \{y\})$.
Therefore, we get $\hullc(S')=\hullc(S'\setminus \{x\})\cup\hullc(S'\setminus \{y\})$.\\
Similarly, using the same argument as above, we get $\hullc(S')=\hullc(S'\setminus \{x'\})\cup\hullc(S'\setminus \{y'\})$. Therefore, $S'$ is $E$-dependent in $G$.

   \item[(b)] Suppose $S'$ contains three vertices from a block $B_i$, where $i\in \{1,2,\dots,l\}$. Let $\{u_i,u_i',u_i''\}\subseteq V(B_i)$. Assume that the vertices  $u_i,u_i',u_i''$ form a path in $G$. Without loss of generality, assume that $u_iu_i'u_i''$ is a path in $G$. Then, $\hullc(S'\setminus \{u_i\})=\hullc(S'\setminus \{u_i''\})= \hullc(S')$. Suppose  $u_i,u_i',u_i''$ form a $K_3$ in $G$. Then, $\hullc(S'\setminus \{u_i\})=\hullc(S'\setminus \{u_i'\})=\hullc(S'\setminus \{u_i''\})=\hullc(S')$ and $S'$ is $E$-dependent in $G$.
   
\end{itemize}

\textbf{Case 2:} $S'$ does not contain vertices from some blocks of $G$.

 Since $|S'|= l+2$ and $S'$ cannot contain vertices from some block of $G$, we have the following possibilities. Suppose $S'$ contains at least four vertices from some block of $G$. It is clear from Case 1(b) that $S'$ cannot contain three vertices from any block of $G$. Therefore, it follows that no block of $G$ can contain at least four vertices in $S'$. Suppose $S'$ contains two vertices from at least three blocks of $G$.  
    Then, from {Case 1(a)}, we have proved that $S'$ cannot contain two vertices from two distinct blocks of $G$ and hence we get  $S'$ is $E$-dependent. Suppose $S'$ contains two vertices from the block $B_i$ and three vertices from the block $B_j$. By combining ${(a)}$ and ${(b)}$ of {Case 1}, we get $S'$ is $E$-dependent. Hence, the maximum $E$-independent set of $G$ can contain only $l+1$ vertices and hence the exchange number, $e_{cc}(G)=l+1$. 
    
    \item[(ii)]  Suppose that $G$ contains blocks that satisfy either the vertex-separation property or the edge-vertex property. Assume, for a moment, that there exists at least one block $B_k;1 < k < l $ satisfying the vertex- separation property in $G$. Let $S\subseteq V(G)$ be such that $S=\{u_1,u_1',u_2,u_3,\dots,u_k,u_k',u_{k+1},u_{k+2},\dots,
 u_l\}$ such that $u_i \in V(B_i); i\in \{1,2,\dots,l\}, u_1'\in V(B_1), u_k'\in V(B_k)$, $ u_1u_1'\in E(G)$, $ u_k \in V(B_{k})\cap N(c_k), u_k' \notin C(G)\cup N(c_k)\cup N(u_k)$ where, $c_k$ is the cut vertex between the blocks $B_{k-1}$ and $B_k$ and  $u_i\in N(V(B_{i-1}))\cap V(B_i)$ for $i \in \{2,3,\dots,l\}$. Then, from Fig \ref{fig 7}, we have the following properties: $\hullc(S \setminus \{u_1\})=S \setminus \{u_1\},\hullc(S \setminus \{u_1'\})=S \setminus \{u_1'\},\hullc(S \setminus \{u_k'\})=V(G)$ and $\displaystyle\bigcup_{i\neq k,i=2}^l\hullc(S\setminus \{u_i\})=\displaystyle\bigcup_{i=1}^{l-1} V(B_i)\cup\{u_l\}$. Then, $\hullc(S\setminus \{u_k'\})\not\subseteq\displaystyle\bigcup_{a\in S\setminus\{u_k'\}}\hullc(S\setminus \{a\})$ and thus, $S$ is an $E$-independent set containing $l+2$ vertices.\\

\begin{figure}[H]
    \centering
\includegraphics[width=150mm,scale=1]{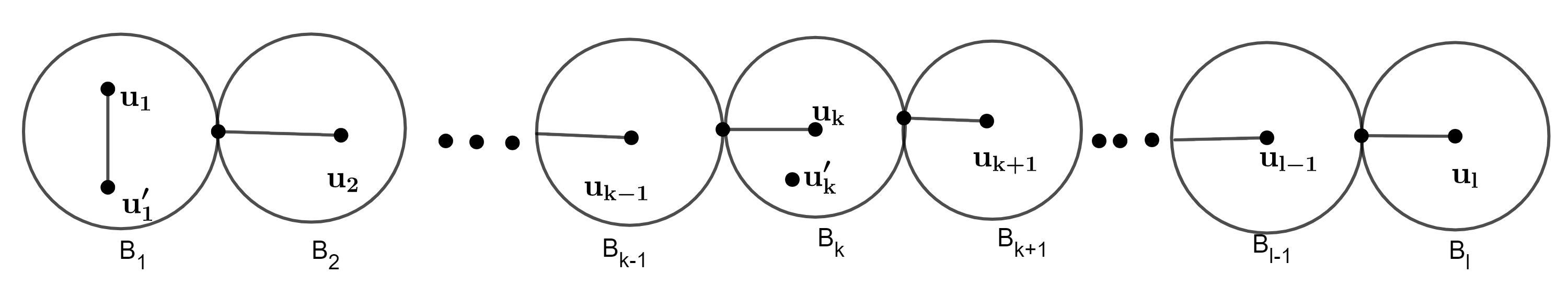}

    \caption{A chordal graph $G$ containing blocks $B_1,B_2,\dots, B_l$ and vertices $u_1,u_1',u_2,\dots,u_{k-1},u_k,u_k',\dots,u_l$ with $\{u_1,u_1'\}\subseteq V(B_1)$,  $u_1u_1'\in E(G), u_k'\in V(B_k)$ and $u_i\in N(V(B_{i-1}))\cap V(B_i)$, for $2\leq i\leq l$ such that $\{u_k,u_k'\}$ satisfying the vertex-separation property with $u_k'$ as the pivot of $S$.}
    \label{fig 7}
  \centering
   \end{figure} 

Now assume that one of the end block of $G$, say $B_1$, satisfies the edge-vertex property. Let $S \subseteq V(G)$ be such that $S=\{u_1,u_1',u_1'',u_2,u_3,\dots,u_l\}$ where $\{u_1,u_1',u_1''\}\subseteq V(B_1)$ with $u_1u_1'\in E(G), u_1''\in \hullc(u_1,u_1')$ and 
 $u_i\in N(V(B_{i-1}))\cap V(B_i)$ for $2 \leq i\leq l$. Then, the following properties holds: (See Fig \ref{fig 8} ) $\hullc(S \setminus \{u_1\})=S \setminus \{u_1\}, \hullc(S \setminus \{u_1'\})=S \setminus \{u_1'\},\hullc(S \setminus \{u_1''\})=V(G) $ and $\displaystyle\bigcup_{i=2}^l\hullc(S\setminus \{u_i\})=\displaystyle\bigcup_{i=1}^{l-1}V(B_i)\cup\{u_l\}$. Therefore, $\hullc(S\setminus \{u_1''\}) \setminus \displaystyle\bigcup _{a \in S \setminus \{u_1''\}}\hullc(S \setminus \{a\})\neq \phi$ implies that $S$ is an $E$-independent set containing $l+2$ vertices.

\begin{figure}[H]
    \centering
\includegraphics[width=100mm,scale=1]{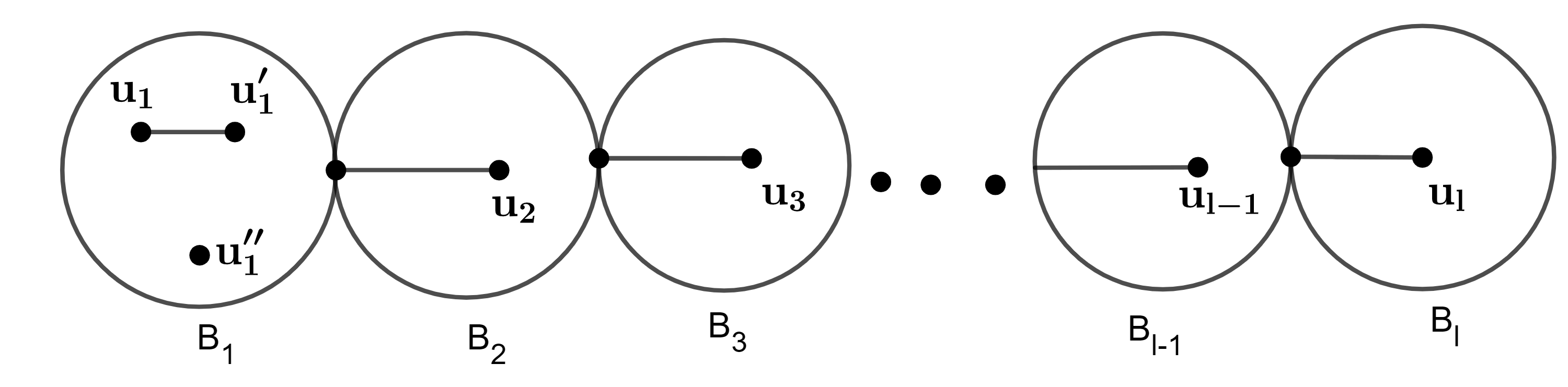 }

    \caption{A chordal graph $G$ containing blocks $B_1,B_2,\dots,B_l$ and vertices $u_1,u_1',u_1'',u_2,\dots,u_l$ such that $\{u_1,u_1',u_1''\}\subseteq V(B_1), u_i\in N(V(B_{i-1}))\cap V(B_i)$ for $2 \leq i\leq l$ with $\{u_1,u_1',u_1''\}$ satisfy the edge-vertex property. Here, $u_1''$ is the pivot of $S$.}
    \label{fig 8}
  \centering
   \end{figure} 

 Now, it remains to show that there does not exists an $E$-independent set of size larger than $l+2$. Assume that $G$ contains an $E$-independent set $S'$ with $|S'|\geq l+3$. Let us consider  $|S'|=l+3$.\\
 
 \end{enumerate} 
     \textbf{Case 1:} $S'$ contains at least one vertex from each block of $G$.
     
      \begin{itemize}
        
    \item[(a)] Suppose that $S'$ contains four vertices from some block of $G$. Let $\{u_x,u_x',u_y,u_y'\}\subseteq S'\cap V(B_i)$, for $ 1 \leq i \leq l$. If $S'$ is the set of all independent vertices of $G$, from Remark \ref{E-dependentexchage}, $S'$ is $E$-dependent. If $\{u_x,u_x',u_y,u_y'\}$ are pairwise adjacent in $G$, then for each $v\in \{u_x,u_x',u_y,u_y'\}$,  $\hullc(S'\setminus \{v\})=\hullc(S')$. Assume that the vertices $u_x,u_x',u_y,u_y'$ form a path in $G$. Without loss of generality, assume that $u_xu_x'u_yu_y'$ is a path of $G$. Then, $\hullc(S'\setminus\{u_x\})=\hullc(S'\setminus\{u_y'\})=\hullc(S')$. Suppose that any three-vertex subset of $\{u_x,u_x',u_y,u_y'\}$ induces a $K_3$ in $G$. Without loss of generality, let $u_x,u_x',u_y$ form a $K_3$ in $G$. Then, $\hullc(S'\setminus \{u_x\})=\hullc(S'\setminus \{u_x'\})=\hullc(S'\setminus \{u_y\})=\hullc(S')$.

\item[(b)] Suppose that $S'$ contains two vertices from three blocks $B_i,B_j$ and $B_k$ of $G$. Let $x,x',y,y',z,z'\in S'$ with $\{x,x'\} \subseteq V(B_i), \{y,y' \}\subseteq V(B_j)$ and $\{z,z'\} \subseteq V(B_k)$. If all the vertices of $S'$ are pairwise independent, then $S'$ is $E$-dependent. Now, assume that the vertices of $S'$ are not independent. Let $v\in \hullc(S')$. We have to prove $ v \in\hullc(S'\setminus \{x\})\cup\hullc(S'\setminus \{y\}\cup\hullc(S'\setminus \{z\})$.  

\textbf{Subcase 1:} Suppose $v \in \hullc(\Gamma_{B_i}{(x,x')} \cup\Gamma_{B_j}{(y,y')}\cup\Gamma_{B_k}{(z,z')})$.
 If $\hullc(\Gamma_{B_i}{(x,x')} \cup \Gamma_{B_j}{(y,y')}\cup\Gamma_{B_k}{(z,z')})=\Gamma_{B_i}{(x,x')} \cup \Gamma_{B_j}{(y,y')}\cup\Gamma_{B_k}{(z,z')}$, 
then for any $v \in \Gamma_{B_i}{(x,x')},$ we get $ v\in \hullc(S'\setminus\{y\}\cup\hullc(S'\setminus\{z\})$, for $v \in \Gamma_{B_j}{(y,y')},$ then $ v \in \hullc(S'\setminus\{x\}\cup \hullc(S'\setminus\{z\})$ and for $v \in \Gamma_{B_j}{(z,z')}$, we get $v \in \hullc(S'\setminus\{x\})\cup\hullc(S'\setminus\{y\})$. Therefore, $v \in \hullc(S'\setminus \{x\})\cup\hullc(S'\setminus \{y\})\cup\hullc(S'\setminus \{z\})$.
 Now, assume that $\hullc(\Gamma_{B_i}{(x,x')} \cup \Gamma_{B_j}{(y,y')}\cup\Gamma_{B_k}{(z,z')})\neq\Gamma_{B_i}{(x,x')} \cup \Gamma_{B_j}{(y,y')}\cup \Gamma_{B_k}{(z,z')}$. Let $v\in \hullc(\Gamma_{B_i}{(x,x')} \cup \Gamma_{B_j}{(y,y')})\setminus\Gamma_{B_i}{(x,x')} \cup \Gamma_{B_j}{(y,y')}$ and $\Gamma_{B_k}(z,z')\subseteq \Gamma_{B_i}{(x,x')} \cup \Gamma_{B_j}{(y,y')}$. Then, $v\in\hullc(S'\setminus\{z\})\cup\hullc(S'\setminus\{z'\})$. Suppose $\Gamma_{B_k}(z,z')\not\subseteq \Gamma_{B_i}{(x,x')} \cup \Gamma_{B_j}{(y,y')}$. Then, $v\in\hullc(S'\setminus\{z\})\cup\hullc(S'\setminus\{z'\})$. Similarly, using the same argument, we get $\hullc(S')=\hullc(S'\setminus \{x'\})\cup\hullc(S'\setminus \{y'\}\cup\hullc(S'\setminus \{z'\})$ and thus $S'$ is $E$-dependent in $G$. 

\textbf{Subcase 2:} Suppose $v \in \hullc(S')\setminus \hullc(\Gamma_{B_i{(x,x')}} \cup\Gamma_{B_j{(y,y')}}\cup\Gamma_{B_k{(z,z')}})$.
Then, $v \in \hullc(S'\setminus \{x\})$ or $v \in \hullc(S'\setminus \{y\})$ or $v \in \hullc(S'\setminus \{z\})$. Therefore, $v \in \hullc(S'\setminus \{x\})\cup\hullc(S'\setminus \{y\}\cup\hullc(S'\setminus \{z\})$.
Similarly, we get $\hullc(S')=\hullc(S'\setminus\{x'\}\cup\hullc(S'\setminus\{y'\}\cup\hullc(S'\setminus\{z'\}$ and hence $S'$ is $E$-dependent.

\item[(c)] Suppose $S'$ contains three vertices from the block $B_i$ and two vertices from $B_j$. If all the vertices of $S'$ are pairwise disjoint, then $S'$ is $E$-dependent. If $\{u,u',u''\}$ are pairwise adjacent, then for any $v \in \{u,u',u''\}$, $\hullc(S'\setminus\{v\})=\hullc(S')$. Assume that the vertices $u,u',u''$ form a path in $B_i$. Let $uu'u''$ be a path in $G$. Then, $\hullc(S'\setminus\{u\})=\hullc(S'\setminus\{u''\})=\hullc(S')$. If both $V(B_i)\cap S'$ and $V(B_j)\cap S'$ contains  independent vertices, then $\{u,u',u'',v,v'\}\subseteq V(G_{B})$ and thus $\hullc(S'\setminus\{u\})\cup\hullc(S'\setminus\{v\})=\hullc(S'\setminus\{u'\})\cup\hullc(S'\setminus\{v'\})=\hullc(S')$. In the same way, if $V(B_i)\cap S'$ contains pairwise disjoint vertices, again $S'$ is $E$-dependent. Suppose that $\{u,u',u''\}\subseteq V(B_i)$ satisfy the edge-vertex property and $\{v,v'\}\subseteq V(B_j)$ has the vertex-separation property. Then, both $u''$ and $v'$ will be the pivots of $S'$ and, therefore, $\hullc(S'\setminus\{u''\})=\hullc(S'\setminus\{v'\})=\hullc(S')$. Thus $S'$ is $E$-dependent in $G$.
 \end{itemize}

 \textbf{Case 2:} $S'$ does not contain vertices from some blocks of $G$.\\
 
Suppose $S'$ contains at least four vertices from a block of $G$. Then, in {Case 1}, it is proved that $S'$ is $E$-dependent if it contains four vertices from some blocks of $G$. Suppose $S'$ contains two vertices from at least two blocks of $G$. This is also not possible from {Case 1} and hence it is $E$-dependent. Suppose $S'$ contains three vertices from at least two blocks of $G$, which is also not possible from the previous cases and hence $E$-dependent.

 Hence, the maximum $E$-independent set of $G$ contains only $l+2$ vertices and therefore, the exchange number, $e_{cc}(G)=l+2$. 
\end{proof}

For the next result, we need two more definitions.

\begin{definition}
    A maximal chain of blocks is defined as a chain which cannot be extended by adding further blocks to it.
\end{definition}
    
    \begin{definition}
        A block which is not isomorphic to $K_2$ is called a non-$K_2$ block.
    \end{definition}

\begin{theorem}\label{theo:non_K2_chain}
    Let $G$ be a chordal graph containing $K_2$ as an induced subgraph, with all of its blocks lying on a single chain. If $l$ denotes the maximum number of blocks in the longest non-$K_2$ chain of $G$, then $e_{cc}(G)=l+2$.
\end{theorem}

\begin{proof}

    Let $G$ be a chordal graph containing $K_2$ as an induced subgraph and all the blocks of $G$ are lying on a single chain. Assume that no block of $G$ satisfies either the vertex-separation property or the edge-vertex property. Suppose $G$ is a tree. Since $G$ contains only $K_2$ as an induced subgraph, $l=0$ and from the Remark \ref{exchangecycle}, $e_{cc}(G)=2$. Let $B_1, B_2,\dots,B_l$ be the blocks in the longest chain which does not contain $K_2$ in $G$. Consider $S_l=\{u_1,u_1',u_2,u_3,\dots,u_l\}$, where $\{u_1,u_1'\} \subseteq V(B_1)$ with $u_1u_1' \in E(G)$ and $u_i \in N(V(B_{i-1}))\cap V(B_i) $ for $i\in \{2,3,\dots,l\}$.
     Then, from the proof of Theorem \ref{chordalexchangenumber} (i), we get $S_l$ is an $E$-independent set of size $(l+1)$. Now, consider the set $S=S_l \cup \{v\}$, where $v\in V(G)\setminus\displaystyle\bigcup_{i=1}^l V(B_i)$.
     Then the following properties hold. 
      $\hullc(S\setminus\{v\})=\displaystyle\bigcup_{i=1}^l V(B_i), \displaystyle\bigcup_{a\in S\setminus \{v\}}\hullc(S \setminus \{a\})=\displaystyle\bigcup_{i=1}^{l-1}V(B_i)\cup\{u_l\}$, where $u_l\in V(B_l)$.
      Then, $\hullc(S\setminus \{v\}) \not \subseteq \displaystyle\bigcup_{a\in S\setminus \{v\}}\hullc(S\setminus \{a\})$ and $S$ is an $E$-independent set containing $l+2$ vertices. \\

      Suppose that $G$ contains blocks satisfying the vertex-separation property or the edge-vertex property. Let $B_1, B_2,\dots,B_l$ be the blocks in the longest non-$K_2$ chain of $G$. Then, there exists at least one block $B_k;1 < k < l $, that satisfies the vertex- separation property in $G$. Let $S_1\subseteq V(G)$ be such that
  $S_1=\{u_1,u_1',u_2,u_3,\dots,u_k,u_k',u_{k+1},u_{k+2},\dots,
 u_l\}$. If one of the end block of $G$, say $B_1$, holds the edge-vertex property, then consider the set $S_2 \subseteq V(G)$ such that $S_2=\{u_1,u_1',u_1'',u_2,u_3,\dots,u_l\}$, where all the vertices of $S_1$ and $S_2$ are chosen as in the proof of the theorem \ref{chordalexchangenumber} (ii). Moreover, in the same theorem, we have proved that both $S_1$ and $S_2$ are $E$-independent sets with $l+2$ vertices.\\ 
      
Now, we must show that there is no $E$-independent set of size greater than $l+2$. Assume that there exists an $E$-independent set $S'$ with $|S'|= l+3$. 
    
\begin{enumerate}

\item[(i)] Suppose $S'$ contains vertices from two distinct chains of blocks.

Let $L_1$ and $L_2$ be any two distinct maximal chains of blocks in $G$. Assume that $S'$ contains $l+2$ vertices from the chain $L_1$ and one vertex $v$ from $L_2$. Assume that no block of $L_1$ holds the vertex-separation property or the edge-vertex property. Then, from Theorem \ref{chordalexchangenumber} (i), it is proved that no $E$-independent set can contain more than $l+1$ vertices from a non-$K_2$ chain of blocks and thus $S'$ is $E$-dependent in $G$. Now, assume that $L_1$ contains blocks satisfying one of the above mentioned properties. For $S\subseteq V(L_1)$, let $S'=S\cup \{v\}$. Then, $\hullc(S'\setminus \{p\})\setminus\displaystyle\bigcup_{a\in S'\setminus\{p\}}\hullc(S'\setminus \{a\})=\phi$, where $p\in \{u_k',u_1''\}$ is a pivot of $S'$. Therefore, $S'$ is $E$-dependent.
Now, let $S'$ contains $l+1$ vertices from $L_1$ and two vertices from $L_2$. If $S'$ is the set of all independent vertices of $G$, then $S'$ is $E$-dependent. Otherwise, let $u,u',v,v'\in S'$ with $\{u,u'\}\subseteq V(L_1)$ and $\{v,v'\}\subseteq V(L_2)$. Then, $\hullc(v,v')$ is contained in both $\hullc(S'\setminus \{u\})$ and $\hullc(S'\setminus \{u'\})$. Similarly, $\hullc(S'\cap V(L_1))$ is contained in both $\hullc(S'\setminus\{v\})$ and $\hullc(S'\setminus\{v'\})$. Therefore, $\hullc(S'\setminus \{u\})\cup\hullc(S'\setminus \{v\})=\hullc(S'\setminus \{u'\})\cup\hullc(S'\setminus \{v'\})=\hullc(S')$ and $S'$ is $E$-dependent.

\item[(ii)]Suppose $S'$ contains vertices from more than two distinct chains of blocks.

Assume that there exists three distinct chains $L_1, L_2$ and $L_3$ with $u \in S'\cap V(L_1), v \in S'\cap V(L_2)$ and $w \in S'\cap V(L_3)$. Then, $\hullc(S'\cap V(L_2))\cup\hullc(S'\cap V(L_3))\subseteq\hullc(S'\setminus\{u\}),\hullc(S'\cap V(L_1))\cup\hullc(S'\cap V(L_3))\subseteq\hullc(S'\setminus\{v\})$ and $\hullc(S'\cap V(L_1))\cup\hullc(S'\cap V(L_2))\subseteq\hullc(S'\setminus\{w\})$. Therefore, $\hullc(S'\setminus \{u\}) \cup \hullc(S'\setminus \{v\})=\hullc(S'\setminus \{v\}) \cup \hullc(S'\setminus \{w\}) = \hullc(S'\setminus \{u\}) \cup \hullc(S'\setminus \{w\}) = \hullc(S')$, which implies that $S'$ is $E$-dependent.

Hence, $G$ does not contain an $E$-independent set of size greater than $l+2$, and therefore $e_{cc}(G)=l+2$.
\end{enumerate}
\end{proof}

We conclude this section by discussing the time to compute the exchange number with the provided formulas of Theorems~\ref{chordalexchangenumber} and~\ref{theo:non_K2_chain}. For that, let $G$ be a chordal graph with $n$ vertices and $m$ edges. 

We remark that for chordal graphs, the blocks can be identified in linear time using a perfect elimination ordering, obtained for instance by the Lexicographic Breadth-First Search (LexBFS) algorithm~\cite{golumbic2004algorithmic}. 
Moreover, determining whether a block $B$ of $G$ satisfies the edge-vertex property (recall Definition~\ref{def:edge_vertex_property}) relies in checking edges and vertices whose distances are greater than two, which can be done by scanning all edges for every vertex, giving $O(m \cdot n)$ time. 
In addition, verifying whether a block $B$ satisfies the vertex-separation property (remind Definition~\ref{def:vertex_separation_property}) requires testing the existence of two vertices $x, y \in V(B)$ with respect to another vertex $c$ under specific neighborhood constraints, which results in $O(n^3)$ time. 
All these checks can therefore be performed in polynomial time, and consequently, according to Theorems~\ref{chordalexchangenumber} and~\ref{theo:non_K2_chain}, the exchange number under such conditions can be determined in polynomial time as well.

\section{Exchange Number in Graph Products}\label{section-product}

 In this section, we derive a lower bound for the exchange number with respect to cycle convexity in the Cartesian product of graphs, and obtain exact values for the strong and lexicographic products.
 
\begin{theorem}\label{cartexchange}
 Let $G$ and $H$ be two nontrivial connected graphs, then \\ $e_{cc}(G\Box H)\geq (e_{cc}(G)-1)(e_{cc}(H)-1)+1$.

\end{theorem}
\begin{proof}

Assume that $e_{cc}(G)= e_{cc}(H)=2$. Then, any two vertices in $G\Box H$ is an $E$-independent set in $G\Box H$, and hence $e_{cc}(G\Box H)\geq (e_{cc}(G)-1)(e_{cc}(H)-1)+1=(2-1)(2-1)+1=2$.
If $e_{cc}(G)>2$ and $ e_{cc}(H)=2$, then by the definition of Cartesian product of graphs, for any $E$-independent set $S$ in $G$, $S\times \{u\}$ (for any $u\in V(H)$) is an $E$-independent set in $G\Box H$. Therefore, $e_{cc}(G\Box H)\geq (e_{cc}(G)-1)(e_{cc}(H)-1)+1=(e_{cc}(G)-1)(2-1)+1=e_{cc}(G)$. Similarly, if $e_{cc}(G)=2$ and $ e_{cc}(H)>2$, then $e_{cc}(G\Box H)\geq (e_{cc}(G)-1)(e_{cc}(H)-1)+1=(2-1)(e_{cc}(H)-1)+1=e_{cc}(H)$.

Assume that $e_{cc}(G), e_{cc}(H)>2$.
Let $U$ and $V$ be the $E$-independent sets in $G$ and $H$ respectively with $|U|=e_{cc}(G)$ and $|V|=e_{cc}(H)$. Let $g$ and $h$ be the pivots of $U$ and $V$ respectively. Then, there exists some $x\in \hullc(U \setminus  \{g\})$, but $\displaystyle x\notin \bigcup_{a\in U \setminus  \{g\}}\hullc(U \setminus  \{a\}) $ and some $y\in \hullc(V\setminus \{h\})$, but $\displaystyle y\notin \bigcup_{b\in V \setminus  \{h\}}\hullc(V \setminus  \{b\}) $. Let $W=(U\setminus \{g\})\times (V\setminus \{h\})\cup \{(g,h)\}$. Then, $|W|=(e_{cc}(G)-1)(e_{cc}(H)-1)+1$. 

\noindent {\bf Claim}: $(g,h)$ is the pivot of $W$.

Since $x\in \hullc(U\setminus \{g\})$, for any $h'\in V\setminus \{h\}$, $(x,h')\in \hullc((U\setminus \{g\})\times \{h'\})$. Therefore, $\{x\}\times (V\setminus \{h\})\subseteq \hullc((U\setminus \{g\})\times (V\setminus \{h\}))\subseteq \hullc(W\setminus \{(g,h)\})$.

Since $y\in \hullc(V\setminus \{h\})$, $(x,y)\in \hullc(\{x\}\times (V\setminus \{h\}))\subseteq \hullc(W\setminus \{(g,h)\})$. 

Let $(a,b)\in \hullc(W\setminus \{(g,h)\})$. Then, $(x,b)\notin \hullc(W\setminus \{(a,b)\})$ and $(a,y)\notin \hullc(W\setminus \{(a,b)\})$, which implies  that $(x,y)\notin \hullc(W\setminus \{(a,b)\})$.  Therefore, $\displaystyle (x,y)\in \hullc(W\setminus \{(g,h)\})\setminus \bigcup_{(a,b)\in W\setminus \{(g,h)\}}\hullc(W\setminus \{(a,b)\}) $.  So $W$ is an $E$-independent set in $G\Box H$ having cardinality $(e_{cc}(G)-1)(e_{cc}(H)-1)+1$. Therefore, $e_{cc}(G\Box H)\geq (e_{cc}(G)-1)(e_{cc}(H)-1)+1$.
\end{proof}

The following theorems gives the exchange number of products of certain graph classes.

\begin{theorem}
    For $n,m\geq 2$, $e_{cc}(K_m\Box P_n)=n+1$.
\end{theorem}
\begin{proof}
    Let $G=K_m\Box P_n$ and let $\{v_1,v_2,\ldots, v_n\}$ be the vertices of $P_n$. Let $g_1,g_2\in E(K_n)$. Consider the set $W=\{(g_2,v_1),(g_1,v_1),(g_1,v_2),\ldots,(g_1,v_{n-1}),(g_1,v_n)\}$. 

    {\bf Claim:} $(g_1,v_n)$ is a pivot of $W$.

    Since $(g_1,v_1),(g_2,v_1)\in W\setminus \{(g_1,v_n)\}$ and $\hullc((g_1,v_1),(g_2,v_1)) = V(K_m)\times \{v_1\}$, $V(K_m)\times \{v_1\}\subseteq W\setminus \{(g_1,v_n)\}$. The vertices $(g_1,v_1),(g_1,v_2),(g_2,v_2)$ and $(g_2,v_1)$ form a cycle in $K_m\Box P_n$ with $(g_1,v_1),(g_1,v_2),(g_2,v_2)\in W\setminus \{(g_1,v_n)\}$, then $(g_2,v_2)\in \hullc(W\setminus \{(g_1,v_n)\})$. Again $(g_2,v_2)$, $(g_1,v_2)$, $(g_1,v_3)$, and $(g_2,v_3)$ form a cycle in $G$ with $(g_1,v_1),(g_1,v_2),(g_2,v_2)\in \hullc(W\setminus \{(g_1,v_n)\})$, then $(g_2,v_2)\in W\setminus \{(g_1,v_n)\}$. By continuing like this we get $\{g_2\}\times (V(P_n)\setminus \{(g_1,v_n)\})\subseteq W\setminus \{(g_1,v_n)\}$. For any $g\in V(K_n)$ with $g\neq g_1$, the vertices $(g,v_1), (g_1,v_1),(g_1,v_2)$ and $(g,v_2)$ form a cycle in $G$ with $\{(g,v_1), (g_1,v_1),(g_1,v_2)\}\subseteq \hullc(W\setminus \{(g_1,v_n)\}$, we get $(g,v_2)\in \hullc(W\setminus \{(g_1,v_n)\}$. By continuing this process we get $\{g\}\times (V(P_n)\setminus \{(g_1,v_n)\})\subseteq \hullc(W\setminus \{(g_1,v_n)\}$. Therefore, $\hullc(K_m\times (V(P_n)\setminus \{v_n\})) \subseteq \hullc(W\setminus \{(g_1,v_n)\})$.

Now $\hullc(W\setminus \{(g_2,v_1)\}) = \{(g_1,v_1),(g_1,v_2),\ldots,(g_1,v_{n-1}),(g_1,v_n)\}$. Let $(g_1,v_i)\in W\setminus \{(g_1,v_n)\}$. By applying the similar arguments as above we get $\hullc(W\setminus \{(g_1,v_i)\}) =V(K_n)\times \{v_1,v_2,\ldots, v_i\}\cup \{v_{i+1},v_{i+2},\ldots, v_n\}$, i.e., for any $(g_1,v_i)\in W\setminus \{(g_1,v_n)\}$, $K_m\times \{v_{n-1}\}\cap \hullc(W\setminus \{(g_1,v_i)\})=\{(g_1,v_{n-1})\}$. That means $(V(K_m)\times \{v_{n-1}\})\setminus \{g_1,v_{n-1}\}\nsubseteq \hullc(W\setminus \{(g_1,v_i)\})$. But $\hullc(K_m\times (V(P_n)\setminus \{v_n\})) \subseteq \hullc(W\setminus \{(g_1,v_n)\})$. Then $(g_1,v_n)$ is a pivot of $W$ and $W$ is an $E$-independent set in $G$ having cardinality $n+1$. Therefore, $e_{cc}(G)\geq  n+1$.

    Let $S\subseteq V(G)$ having cardinality $n+2$. 
    
Assume that $(g_1,v_j),(g_2,v_j),(g_3,v_j)\in V(K_n)\times \{v_j\})$, for some $(g_1,v_j),(g_2,v_j),(g_3,v_j)\in S$. Since $g_1,g_2,g_3\in V(K_m)$, $\hullc(S\setminus \{(g_1,v_j)\}) = \hullc(S\setminus \{(g_2,v_j)\}) = \hullc(S\setminus \{(g_3,v_j)\})$ and hence $S$ is an $E$-dependent set in $G=K_m\Box P_n$.

    Since $S$ contains $n+2$ vertices, at least two $K_m$-layers contain two vertices each from $S$. Assume for a moment that $(g,v_r)$ is a pivot of $S$ with $S\cap V(K_m^{v_r})=(g,v_r)$. By Lemma~\ref{lemma:e_ind_set_cc}, $G[\hullc(S\setminus \{(g,v_r)\})]$ is connected and hence for any $(g', h)\in S\setminus \{(g,v_r)\}$, there exists some $(g',h')\in S\setminus \{(g,v_r)\}$ with $gg'\in E(P_n)$. Assume for a moment that exactly two $K_m$ layers has two vertices each from $S$, let the layers be $K_m^{h_i}$ and $K_m^{h_j}$. Now arrange the vertices of $S\setminus \{(g,v_r)\}$ as follows. $S\setminus \{(g,v_r)\}=\{(g_1,h_1),(g_2,h_2), \ldots, (g_i,h_i),(g_i',h_i),(g_{i+1},h_{i+1}),\ldots, (g_j,h_j),(g_j',h_j),$ $(g_{j+1},h_{j+1}),\ldots, (g_s,h_s)\}$ with $h_1,h_2,\ldots,h_s$ is a path in $P_n$. Then, $\langle\{(g_i,h_i),(g_i',h_i)\})=V(K_m^{h_i})$ and $\langle\{(g_j,h_j),(g_j',h_j)\})=V(K_m^{h_j})$. As in the first part above, we can enumerate that $\hullc(S\setminus \{(g,v_r)\})=V(K_m)\times \{h_1,h_2,\ldots, h_s\}$. Also $\hullc(S\setminus (g_i',h_i))=\hullc(S\setminus (g_j',h_j))=V(K_m)\times \{h_1,h_2,\ldots, h_s\}=\hullc(S\setminus \{(g,v_r)\})$. Therefore, $S$ is an $E$-dependent set in $K_m\Box P_n$.

    If $S= (g_1,v_1),(g_2,v_2),\ldots,(g_i,v_i),(g_i',v_i),\ldots, (g_t,v_t)$ with $(g_i,v_i)$ is a pivot of $S$, then by Lemma~\ref{lemma:e_ind_set_cc}, $G[\hullc(S\setminus \{(g_i,v_i)\})]$ is connected and hence $v_1,v_2,\ldots,v_i,\ldots,v_t$ is a path in $P_n$. Then, $\hullc(S\setminus \{(g_i,v_i)\})=\hullc(S\setminus \{(g_i,v_i)\})=V(K_m)\times \{v_1,v_2,\ldots, v_i,\ldots, v_t\}$. Then, $S$ is an $E$-dependent set in $G=K_m\Box P_n$. 

    Hence any set having cardinality $n+2$ is an $E$-dependent set in $K_m\Box P_n$. Therefore, $e_{cc}(K_m\Box P_n)=n+1$. 
\end{proof}

\begin{theorem}
    For $n,m\geq 2$, $e_{cc}(P_m\Box P_n)\geq m+n-1$.
\end{theorem}
\begin{proof}
    Let $V(P_m)=\{g_1,g_2,\ldots,g_m\}$ and $V(P_n)=\{h_1,h_2,\ldots,h_n\}$. Consider the set $S=\{(g_1,h_1),(g_1,h_2), \ldots, (g_1,h_n), (g_2,h_n),(g_3,h_n),\ldots, (g_m,h_n)\}$. Then, $|S|=m+n-1$. 

    {\bf Claim:} $(g_1,h_1)$ is a pivot of $S$.

    By the similar argument in the proof of previous theorem, one can find that $\hullc(S\setminus \{(g_1,h_1)\})=V(P_m)\times \{h_2,h_3,\ldots, h_n\}$  and contains $(g_1,h_n)$. For any $(g_1,h_i)\in S$ with $2\leq i\leq n$, $\hullc(S\setminus (g_1,h_i))=V(K_m)\times \{h_{i+1}, h_{i+2},\ldots,h_n\}\cup \{(g_1,h_1),(g_2,h_2),\ldots, (g_1,h_{i-1})\}$ and $(g_1,h_n)\notin \hullc(S\setminus (g_1,h_i))$. For any $(g_j,h_n)\in S$ with $i\leq j\leq m$, $\hullc(S\setminus (g_j,h_n))=\{g_1,g_2,\ldots,g_{j-1}\}\times V(K_n)\cup \{(g_{j+1},h_n),(g_{j+2,h_n}),\ldots,(g_m,h_n)\}$ and $(g_1,h_n)\notin \hullc(S\setminus \{(g_j,h_n)\})$. Now, we can conclude that $\displaystyle (g_1,h_n)\in (S\setminus \{(g_1,h_1)\})\setminus \bigcup_{(a,b)\in S\setminus \{(g_1,h_1)\}}\hullc(S\setminus \{(a,b)\})$. Therefore, $(g_1,h_1)$ is a pivot of $S$.
    
    Hence $S$ is an $E$-independent set in $P_m\Box P_n$ with cardinality $m+n-1$. Which completes the proof.
\end{proof}

\begin{theorem}
    For $n,m\geq 3$, $e_{cc}(K_m\Box K_n)= 3$.
\end{theorem}
\begin{proof}
Let $G = K_m \Box K_n$. Suppose $S \subseteq V(G)$ contains three vertices from a $K_m$-layer, namely $(g_1,h)$, $(g_2,h)$ and $(g_3,h)$. Then, $\hullc(\{(g_1,h), (g_2,h)\})=\hullc(\{(g_2,h), (g_3,h)\})=\hullc(\{(g_1,h), (g_3,h)\})$ and hence $\hullc(S\setminus \{(g_1,h)\}) = \hullc(S\setminus \{(g_3,h)\})$, which in turn $S$ is an $E$-dependent set in $G$. Thus, if $S$ is an $E$-independent set in $G$, then $S$ cannot contain three vertices from any $K_m$-layer or any $K_n$-layer. 

Let $S=\{(g_1,h_1),(g_2,h_2),(g_3,h_3)\}$ with $(g_1,h_1)(g_2,h_2)\in E(G)$ where $(g_3,h_3)$ is not adjacent to either 
$(g_1,h_1)$ or $(g_2,h_2)$. Since $(g_1,h_1)(g_2,h_2)\in E(G)$, either $g_1=g_2$ and $h_1h_2\in E(K_n)$ or $h_1=h_2$ and $g_1g_2\in E(K_m)$. Assume for a moment that $g_1=g_2$ and $h_1h_2\in E(K_n)$. Then, $\hullc(S\setminus \{(g_3,h_3)\})=\hullc(\{(g_1,h_1),(g_2,h_2)\}) =V(^{g_1}K_n)$. But $\hullc(S\setminus \{(g_1,h_1)\})=\{(g_2,h_2),(g_3,h_3)\}$ and $\hullc(S\setminus \{(g_2,h_2)\})=\{(g_1,h_1),(g_3,h_3)\}$. Therefore, $(g_3,h_3)$ is a pivot of $S$ in $G$ and hence $S$ is an $E$-independent set in $G$ with $|S|=3$.

Let $S'\subseteq V(G)$ with $|S'|=4$, and let $S'=\{(g_1,h_1),(g_2,h_2),(g_3,h_3),(g_4,h_4)\}$. If the induced subgraph $G[S']$ contains no edges, then it is $E$-dependent. Assume $(g_1,h_1)(g_2,h_2)\in E(G)$, then $g_1=g_2$ and $h_1h_2\in E(K_n)$ or $h_1=h_2$ and $g_1g_2\in E(K_m)$. Assume $g_1=g_2$ and $h_1h_2\in E(K_n)$. Consider the set $S\setminus \{(g_4,h_4)\}$.$\hullc(\{(g_1,h_1),(g_2,h_2)\})=V(^{g_1}K_n)\subseteq \hullc(S\setminus \{(g_4,h_4)\})$. Then, $(g_3,h_3), (g_1,h_3)\in \hullc(S\setminus \{(g_4,h_4)\})$ and hence $\hullc(\{(g_3,h_3), (g_1,h_3)\}) =V(K_m^{h_3})\subseteq \hullc(S\setminus \{(g_4,h_4)\})$. Now the vertices of the two layers $K_m^{h_3}$ and $^{g_1}K_n$ are in $\hullc(S\setminus \{(g_4,h_4)\})$ and hence the next iteration we get $V(G)\subseteq \hullc(S\setminus \{(g_4,h_4)\})$. Similarly we can prove that $V(G)\subseteq \hullc(S\setminus \{(g_3,h_3)\})$. Therefore, $S$ is an $E$-dependent set in $G$. 

Now we can conclude that $e_{cc}(K_m\Box K_n)= 3$.
\end{proof}

To end this section, we recall that with respect to $\Delta$-convexity, it was shown in \cite{bijo2} that any two adjacent vertices form a hull set in $G \ast H$, where $\ast \in \{\boxtimes, \circ\}$. The same property also holds for cycle convexity. Hence, the following theorems are valid.

\begin{theorem}\label{exstrong}
	Let $G$ and $H$ be two nontrivial connected graphs with at least one of them has diameter greater than two, then $e(G\boxtimes H)= 3$.
\end{theorem}
\begin{theorem}\label{exlex}
	Let $G$ and $H$ be two nontrivial connected graphs. Then \\ $e_{cc}(G\circ H)= \left\{
	\begin{array}{ll}
	
	3, & \hbox{if $G$ has diameter at least two or $H$ has the edge-vertex property;} \\
	2, & \hbox{otherwise.}
	\end{array}
	\right.
	$
\end{theorem}

\section{Concluding Remarks}
\label{sec:conclusion}

In this paper, we have investigated the exchange number in the cycle convexity in graphs. 
We proved that the decision problem of such a parameter is NP-complete even for $K_5$-free graphs and obtained formulas for chordal graphs whose blocks form a single chain and other several graph classes, including trees, cycles, complete, multipartite, and unicyclic graphs. We also characterized graphs with exchange number equals to its number of vertices minus one. Further we established bounds for Cartesian, and equalities for strong, and lexicographic graph products.
As future work, we suggest investigating the computational complexity of determining the exchange number on all chordal graphs, as well as deriving upper bounds for the exchange number on Cartesian product graphs.\\

\section*{Acknowledgements}

Revathy S.Nair acknowledges the financial support from the University of Kerala, for providing the University Junior Research Fellowship (Ac EVI 3217/2025/UOK dated 10/04/2025).\\ 

\section*{Declaration}

The authors declare that they have no known competing financial interests or personal relationships that could have appeared to influence the work reported in this paper. Also declare that no new datasets were generated or analyzed in the work presented. Data sharing is not applicable to this article.

\bibliographystyle{amsplain}
\bibliography{cycle.bib}

\end{document}